\documentclass[11pt,letterpaper]{amsart}
\usepackage[usenames,dvipsnames]{color}
\usepackage{amsthm,amsfonts,amssymb,amsmath,amsxtra}
\usepackage[all]{xy}
\SelectTips{cm}{}
\usepackage{xr-hyper}
\usepackage[pdfpagelabels,pdftex,hidelinks]{hyperref}
\hypersetup{
  colorlinks   = true, 
  urlcolor     = blue, 
  linkcolor    = blue, 
  citecolor   = blue, 
pdftitle={},
  pdfauthor={},
  pdfsubject={},
  pdfkeywords={},
  breaklinks=true,
  bookmarksopen=true,
  bookmarksnumbered=true,
  pdfpagemode=UseOutlines,
  plainpages=false}
\usepackage{cleveref}

\usepackage{verbatim}
\usepackage{caption}

\usepackage{lineno}

\usepackage{mathrsfs}

\RequirePackage{xspace}
\RequirePackage{etoolbox}
\RequirePackage{varwidth}
\RequirePackage{enumitem}
\RequirePackage{tensor}
\RequirePackage{mathtools}
\RequirePackage{longtable}
\RequirePackage{multirow}

\def\ge{\geqslant}
\def\le{\leqslant}
\def\a{\alpha}
\def\b{\beta}
\def\g{\gamma}
\def\G{\Gamma}
\def\d{\delta}
\def\D{\Delta}
\def\L{\Lambda}

\def\s{\sigma}
\def\t{\tau}
\def\th{\theta}
\def\k{\kappa}
\def\l{\lambda}
\def\z{\zeta}

\def\i{^{-1}}

\def\<{\langle}
\def\>{\rangle}

\newcommand{\BA}{\ensuremath{\mathbb {A}}\xspace}

\newcommand{\BC}{\ensuremath{\mathbb {C}}\xspace}

\newcommand{\BE}{\ensuremath{\mathbb {E}}\xspace}
\newcommand{\BF}{\ensuremath{\mathbb {F}}\xspace}
\newcommand{{\BG}}{\ensuremath{\mathbb {G}}\xspace}
\newcommand{\BH}{\ensuremath{\mathbb {H}}\xspace}
\newcommand{\BI}{\ensuremath{\mathbb {I}}\xspace}

\newcommand{{\BK}}{\ensuremath{\mathbb {K}}\xspace}
\newcommand{\BL}{\ensuremath{\mathbb {L}}\xspace}
\newcommand{\BM}{\ensuremath{\mathbb {M}}\xspace}

\newcommand{\BQ}{\ensuremath{\mathbb {Q}}\xspace}
\newcommand{\BR}{\ensuremath{\mathbb {R}}\xspace}

\newcommand{\BT}{\ensuremath{\mathbb {T}}\xspace}
\newcommand{\BU}{\ensuremath{\mathbb {U}}\xspace}
\newcommand{\BV}{\ensuremath{\mathbb {V}}\xspace}
\newcommand{\BW}{\ensuremath{\mathbb {W}}\xspace}
\newcommand{\BX}{\ensuremath{\mathbb {X}}\xspace}

\newcommand{\BZ}{\ensuremath{\mathbb {Z}}\xspace}

\newcommand{\CA}{\ensuremath{\mathcal {A}}\xspace}

\newcommand{\CC}{\ensuremath{\mathcal {C}}\xspace}
\newcommand{\CD}{\ensuremath{\mathcal {D}}\xspace}
\newcommand{\CE}{\ensuremath{\mathcal {E}}\xspace}

\newcommand{\CG}{\ensuremath{\mathcal {G}}\xspace}
\newcommand{\CH}{\ensuremath{\mathcal {H}}\xspace}

\newcommand{\CK}{\ensuremath{\mathcal {K}}\xspace}
\newcommand{\CL}{\ensuremath{\mathcal {L}}\xspace}

\newcommand{\CO}{\ensuremath{\mathcal {O}}\xspace}

\newcommand{\CR}{\ensuremath{\mathcal {R}}\xspace}

\newcommand{\CT}{\ensuremath{\mathcal {T}}\xspace}

\newcommand{\CX}{\ensuremath{\mathcal {X}}\xspace}

\newcommand{\Ad}{{\mathrm{Ad}}}
\newcommand{\ad}{{\mathrm{ad}}}

\DeclareMathOperator{\End}{End}

\DeclareMathOperator{\Hom}{Hom}

\DeclareMathOperator{\Nm}{Nm}

\DeclareMathOperator{\ord}{ord}

\newcommand{\red}{\ensuremath{\mathrm{red}}\xspace}

\newcommand{\SL}{{\mathrm{SL}}}
\DeclareMathOperator{\Spec}{Spec}

\DeclareMathOperator{\tr}{tr}

\newcommand{\ov}{\overline}

\def\brk{{\breve k}}
\def\dw{{\dot w}}

\def\pr{{\rm pr}}
\def\tPhi{\widetilde \Phi}

\def\tD{\widetilde \Delta}
\def\ind{{\rm ind}}
\def\Nm{{\rm Nm}}
\def\aff{{\rm aff}}
\def\bx{{\mathbf x}}

\def\der{{\rm der}}
\def\ov{\overline}
\def\stab{{\rm stab}}

\newtheorem{theorem}{Theorem}
\newtheorem{proposition}[theorem]{Proposition}
\newtheorem{lemma}[theorem]{Lemma}

\newtheorem{corollary}[theorem]{Corollary}

\theoremstyle{definition}
\newtheorem{definition}[theorem]{Definition}
\newtheorem{example}[theorem]{Example}

\newtheorem{remark}[theorem]{Remark}

\numberwithin{equation}{section}
\numberwithin{theorem}{section}

\setitemize[0]{leftmargin=*,itemsep=\the\smallskipamount}
\setenumerate[0]{leftmargin=*,itemsep=\the\smallskipamount}

\renewcommand{\to}{%
   \ifbool{@display}{\longrightarrow}{\rightarrow}%
   }
\let\shortmapsto\mapsto
\renewcommand{\mapsto}{%
   \ifbool{@display}{\longmapsto}{\shortmapsto}%
   }
\newlength{\olen}
\newlength{\ulen}
\newlength{\xlen}
\newcommand{\xra}[2][]{%
   \ifbool{@display}%
      {\settowidth{\olen}{$\overset{#2}{\longrightarrow}$}%
       \settowidth{\ulen}{$\underset{#1}{\longrightarrow}$}%
       \settowidth{\xlen}{$\xrightarrow[#1]{#2}$}%
       \ifdimgreater{\olen}{\xlen}%
          {\underset{#1}{\overset{#2}{\longrightarrow}}}%
          {\ifdimgreater{\ulen}{\xlen}%
             {\underset{#1}{\overset{#2}{\longrightarrow}}}
             {\xrightarrow[#1]{#2}}}}%
      {\xrightarrow[#1]{#2}}
   }
\makeatother
\newcommand{\xyra}[2][]{%
   \settowidth{\xlen}{$\xrightarrow[#1]{#2}$}%
   \ifbool{@display}%
      {\settowidth{\olen}{$\overset{#2}{\longrightarrow}$}%
       \settowidth{\ulen}{$\underset{#1}{\longrightarrow}$}%
       \ifdimgreater{\olen}{\xlen}%
          {\mathrel{\xymatrix@M=.12ex@C=3.2ex{\ar[r]^-{#2}_-{#1} &}}}%
          {\ifdimgreater{\ulen}{\xlen}%
             {\mathrel{\xymatrix@M=.12ex@C=3.2ex{\ar[r]^-{#2}_-{#1} &}}}
             {\mathrel{\xymatrix@M=.12ex@C=\the\xlen{\ar[r]^-{#2}_-{#1} &}}}}}%
      {\mathrel{\xymatrix@M=.12ex@C=\the\xlen{\ar[r]^-{#2}_-{#1} &}}}%
   }
\makeatletter
\newcommand{\xla}[2][]{%
   \ifbool{@display}%
      {\settowidth{\olen}{$\overset{#2}{\longleftarrow}$}%
       \settowidth{\ulen}{$\underset{#1}{\longleftarrow}$}%
       \settowidth{\xlen}{$\xleftarrow[#1]{#2}$}%
       \ifdimgreater{\olen}{\xlen}%
          {\underset{#1}{\overset{#2}{\longleftarrow}}}%
          {\ifdimgreater{\ulen}{\xlen}%
             {\underset{#1}{\overset{#2}{\longleftarrow}}}
             {\xleftarrow[#1]{#2}}}}%
      {\xleftarrow[#1]{#2}}
   }
\newcommand{\isoarrow}{%
   \ifbool{@display}{\overset{\sim}{\longrightarrow}}{\xrightarrow\sim}%
   }

\newcommand{\sm}{{\,\smallsetminus\,}}
\newcommand\cool{\overline{\mathbb{Q}}_\ell}
\newcommand\rar{ \rightarrow }

\newcommand{\obF}{\overline{\BF}_q}
\newcommand{\bfx}{{\mathbf x}}

\usepackage{upgreek}
\makeatletter
\newcommand{\colim@}[2]{%
  \vtop{\m@th\ialign{##\cr
    \hfil$#1\operator@font lim$\hfil\cr
    \noalign{\nointerlineskip\kern1.5\ex@}#2\cr
    \noalign{\nointerlineskip\kern-\ex@}\cr}}%
}
\newcommand{\colim}{%
  \mathop{\mathpalette\colim@{\rightarrowfill@\textstyle}}\nmlimits@
}
\newcommand{\prolim@}[2]{%
  \vtop{\m@th\ialign{##\cr
    \hfil$#1\operator@font lim$\hfil\cr
    \noalign{\nointerlineskip\kern1.5\ex@}#2\cr
    \noalign{\nointerlineskip\kern-\ex@}\cr}}%
}
\newcommand{\prolim}{%
  \mathop{\mathpalette\colim@{\leftarrowfill@\textstyle}}\nmlimits@
}
\makeatother

\begin{document}

\title{Convex elements and cohomology of deep level Deligne-Lusztig varieties}

\author{Alexander B. Ivanov}
\address{Fakult\"at f\"ur Mathematik, Ruhr-Universit\"at Bochum, D-44780 Bochum, Germany.}
\email{a.ivanov@rub.de}

\author{Sian Nie}
\address{Academy of Mathematics and Systems Science, Chinese Academy of Sciences, Beijing 100190, China}

\address{School of Mathematical Sciences, University of Chinese Academy of Sciences, Chinese Academy of Sciences, Beijing 100049, China}
\email{niesian@amss.ac.cn}

\begin{abstract}
We essentially complete a program initiated by Boyarchenko--Weinstein to give a full description of the cohomology of deep level Deligne--Lusztig varieties for elliptic tori, with coefficients in arbitrary non-defining characteristics. We give several applications of our results: we show that the $\phi$-weight part of the cohomology is very often concentrated in a single degree, and is induced from a Yu-type subgroup. Also, we give applications to a previous work of the second author on decomposition of deep level Deligne--Lusztig representations, and to Feng's explicit construction of Fargues--Scholze parameters. Furthermore, a conjecture of Chan--Oi about the Drinfeld stratification is verified as a special case from our results.

\end{abstract}
\maketitle

\section{Introduction}

In \cite{BoyarchenkoW_16} Boyarchenko--Weinstein started a program toward a complete description of the cohomology of certain higher-dimensional varieties over $\overline\BF_q$ equipped with interesting group actions. The varieties they considered came in two disguises: the first were related to special affinoids in Lubin--Tate spaces, and the second were very close to deep level Deligne--Lusztig varieties introduced in \cite{Lusztig_04,CI_MPDL}. In this article we introduce the notion of convex elements in a Weyl group, and essentially complete the program of Boyarchenko--Weinstein for deep level Deligne--Lusztig varieties associated to convex elements. We note that various related/partial results in this direction were obtained in \cite{Chan_siDL,CI_DrinfeldStrat,IvanovNie_24} on deep level Deligne--Lusztig varieties of Coxeter type. 

We also give applications of our results. Most notably, our result verifies the assumption in the work of Feng \cite[\S10]{Feng_24}, which allows to explicitly describe the Fargues--Scholze parameters of many supercuspidal representations in the case of modular coefficients. 

Let $k$ be a non-archimedean local field with residue field $\BF_q$ of characteristic $p$. Let $\breve k$ be the completion of the maximal unramified extension of $k$. Let $F$ denote the Frobenius automorphism of $\breve k$ over $k$. Let $G$ be a reductive group over $k$, which splits over $\breve k$. Let $T$ be a $k$-rational $\breve k$-split elliptic maximal torus of $G$. Let $U$ be the unipotent radical of a $\breve k$-rational Borel subgroup of $G$ containing $T$.

Let ${\bf x}$ be a point in the Bruhat--Tits building of $G$ over $k$. Bruhat--Tits theory attaches to it a parahoric group $\CG_{\bf x}$ over the integers $\CO_k$ of $k$. By the work of Lusztig \cite{Lusztig_04} and Chan and the first author \cite{CI_MPDL}, one associates with $T,U,{\bf x}$ and any $r\geq 0$ a deep level Deligne--Lusztig variety
\[
X_r = X_{T,U,{\bf x},r}
\]
over $\ov\BF_q$, equipped with an action of $\CG_{\bf x}(\CO_k) \times \CT_{\bf x}(\CO_k)$, which factors through a finite Moy--Prasad quotient $\BG_r^F \times \BT_r^F$; here $\CT_{\bf x}$ is the closure of $T$ in $\CG_{\bf x}$ (see \S\ref{notation:2} for more notation and \S\ref{sec:DeligneLusztig} for definition of $X_r$).

We fix a prime number $\ell \neq p$, let $L$ be an algebraic extension of $\BQ_\ell$, $\CO_L$ its ring  of integers and $\lambda$ its residue field, which is an algebraic extension of $\BF_\ell$. We then let $\Lambda$ denote any of the rings $L,\CO_L,\lambda$. For the rest of the introduction, we fix a smooth character $\phi \colon T(k) \to \Lambda^\times$. We assume that $\phi$ admits a Howe factorization \[(G^i,r_i,\phi_i)_{-1\leq i\leq d}\] in the sense of \cite[\S3.6]{Kaletha_19}\footnote{In \emph{loc.cit.} the Howe factorization is only defined for when $\Lambda = \overline \BQ_\ell$, but we may exploit the same definition for any $\Lambda$, cf. \S\ref{sec:Howe}. Note that at least when $\Lambda \cong \cool$, Howe factorizations always exist by \cite[Proposition 3.6.7]{Kaletha_19} when $p$ is neither a bad prime for $G$, nor divides $|\pi_1(G_{\rm der})|$.}. In this article we are interested in studying the \emph{deep level Deligne--Lusztig complex}
\[
R\Gamma_c(X_r,\Lambda)[\phi] := R\Gamma_c(X_r,\Lambda) \otimes^L_{\Lambda\BT_r^F} \phi|_{\BT_r^F} \in  D^b(\BG_r^F{-\rm mod}) \subseteq  D^b(\CG_{\bf x}(\CO_k){-\rm mod}).
\]
Moreover, there is a natural extension \[R\Gamma_c(X_r,\Lambda)[\phi]  \in D^b(Z(k)\CG_{\bf x}(\CO_k){-\rm mod}),\] where $Z$ denote the center of $G$ and  $Z(k)$ acts via the restriction $\phi|_{Z(k)}$.

Suppose for a moment that $\Lambda = \cool$. Then $R\Gamma_c(X_r,\Lambda)[\phi]$ contains the same information as the $\phi$-weight part \[R_{T,U,r}^G(\phi) = \sum_{i\in \BZ} (-1)^i H^i_c(X_r,\cool)[\phi]\] of the equivariant $\ell$-adic Euler characteristic of $X_r$, which is a virtual representation of $Z(k)\CG_{\bf x}(\CO_k)$. In \cite{Chan24} Chan established the inner product formula for the representations $R_{T,U,r}^G(\phi)$, which implies that they are independent of the choice of $U$. In \cite{Nie_24} the second author gave a very explicit decomposition of $R_{T,U,r}^G(\phi)$:
\begin{equation}\label{eq:decomposition_RTUphi}
R_{T,U,r}^G(\phi) = \ind_{Z(k)\CK_\phi(\CO_k)}^{Z(k)\CG_{\bf x}(\CO_k)} \left(\kappa_\phi \otimes R_{T,U,0}^{G^0}(\phi_{-1}) \right) 
\end{equation}
where $\CK_\phi = \CK_{\phi,{\bf x}}$ is a second $\CO_k$-model of $G$ determined by the Howe datum $(G^i,r_i)_{-1 \le i \le d}$, such that $\CK_\phi(\CO_k) \subseteq \CG_{\bf x}(\CO_k)$ is a ``Yu-type subgroup'' \cite{Yu_01},
and $R_{T,U,0}^{G^0}(\phi_{-1})$ is a classical Deligne--Lusztig representation, regarded as a representation of $Z(k)\CK_\phi(\CO_k)$ by inflation. Here, $\kappa_\phi$ is an irreducible representation of $Z(k)\CK_\phi(\CO_k)$ up to a sign, defined in cohomological terms, which is irreducible by \cite[Proposition 1.4]{Nie_24}. In \cite{Yu_01} J.-K. Yu's constructed another representation $\kappa(\phi)$ of $Z(k)\CK_\phi(\CO_k)$ using the Weil--Heisenberg representation. When $q >4$, it is proved in \cite{LN} that $\pm\kappa_\phi$ and $\kappa(\phi)$ differ precisely by the quadratic character of Fintzen--Kaletha--Spice \cite[Theorem 4.1.13]{FKS}.

There is a second variety with much simpler geometric structure,
\[
Z_{\phi,U,r},
\]
also equipped with an action of $\CG_{\bf x}(\CO_k) \times \CT_{\bf x}(\CO_k)$, inflated from $\BG_r^F \times \BT_r^F$ (see \S\ref{sec:comp_with_CS_variety} for definition). It was first considered in special cases by Chen--Stasinski \cite{ChenS_17,ChenS_23} and plays also an important role in \cite{Nie_24}. Due to its simpler geometry, the cohomology of $Z_{\phi,U,r}$ is much easier to describe than that of $X_r$. Let $\CR_{T,U,r}^G(\phi)$ denote the $\phi$-weight part of the equivariant $\ell$-adic Euler characteristic of $Z_{\phi,U,r}$. A major step in \cite{Nie_24} towards \eqref{eq:decomposition_RTUphi} was to show that
\begin{equation}\label{eq:RequalsCR} R_{T,U,r}^G(\phi) = \CR_{T,U,r}^G(\phi) \end{equation}
as virtual $Z(k)\CG_{\bf x}(\CO_k)$-representations. 

\

Now let $\Lambda$ again be arbitrary. Our first main result show that for a suitably chosen $U$ there is isomorphism (up to a precise shift) between the deep level Deligne-Lusztig complexes of $X_{U, r}$ and $Z_{\phi, U, r}$ with coefficients in $\L$. This can be viewed as a graded extension of \eqref{eq:RequalsCR}. To make the choice of $U$, we introduce the notion of \emph{convex elements} in the Weyl group $W$ of $T$ in $G$ (see \S\ref{sec:convex}). This notion is motivated to extend the Steinberg cross-section theorem on Coxeter elements \cite{Steinberg_65}. It also play an essential role in our approach to the Boyarchenko-Weinstein program (cf. \Cref{uniformization}). 

\begin{theorem}\label{thm:RGamma}
Suppose the relative position of $U$ and $F U$ in the Weyl group of $T$ is a convex element of the Weyl group (cf. \S\ref{sec:convex}).  Then there is an isomorphism:
\[
R\Gamma_c(X_r,\Lambda)[\phi] \cong R\Gamma_c(Z_{\phi,U,r},\L)[\phi][2m]
\]
in $D^b(Z(k)\CG_{\bf x}(\CO_k) {-\rm mod})$, for some (explicit) shift $m \in \BZ_{\geq 0}$.

Moreover, if $G^0$ is a standard Levi subgroup with respect to $U$, then $R\Gamma_c(X_r,\Lambda)[\phi]$ is the induction to $Z(k)\CG_{\bf x}(\CO_k)$ of the $Z(k)\CK_\phi(\CO_k)$-complex
\[ (-1)^{n_\phi}\k_\phi[0] \otimes_{\Lambda} R\Gamma_c(\bar Z_{\phi, U, r}^{\BL}, \Lambda)[\phi_{-1}][-n_\phi],\] where $n_\phi$ is an explicit shift (as in Theorem \ref{kappa}), $(-1)^{n_\phi}\k_\phi[0]$ is concentrated in degree $0$, and $R\Gamma_c(\bar Z_{\phi, U, r}^{\BL}, \Lambda)[\phi_{-1}]$ is the cohomology of a classical Deligne--Lusztig variety attached to the special fiber of the closure $(\CG^0)_{\bfx}$ of $G^0$ in $\CG_\bx$.

\end{theorem}

\begin{remark}
    Based on the work \cite{NieTanYu_24} by Tan, Yu and the second author, it is proved in Proposition \ref{standard} that the group $U$ as in \Cref{thm:RGamma}, as well as in its ``moreover part'', always exists.
\end{remark}

The first part of \Cref{thm:RGamma} follows directly from \Cref{thm:concentrates_Howe} and \Cref{prop:comp_XandZ}, and the second part from \Cref{degreewise}. The strategy to Theorem \ref{thm:RGamma} is similar to that in \cite{IvanovNie_24}. It is based on vanishing results on the cohomology of certain local systems related to $X_r$ and $\phi$, in which the notion of convex elements plays an essential role.

The original proof of \eqref{eq:RequalsCR} in \cite{Nie_24} is to compare various inner products of virtual representations. Our proof of Theorem \ref{thm:RGamma} gives a new interpretation of \eqref{eq:RequalsCR} in the convex case.


\

Now we discuss some applications of \Cref{thm:RGamma}. Most importantly, \Cref{thm:RGamma}, along with further results from \cite{Nie_24}, implies that $\phi$-weight part of cohomology of $X_r$ is often concentrated in a single cohomological degree. 

\begin{corollary}[see \Cref{cor:concentration_in_general} for a precise statement] \label{cor:concentration_intro}
Let $T,U$ be as in the ``moreover part'' of \Cref{thm:RGamma}. Assume furthermore that  $\phi_{-1}$ is non-singular for the special fiber of $(\CG^0)_{\bf x}$ in the sense of \cite[Definition 5.15]{DeligneL_76}. 
Then there exists a unique integer $N_\phi \geq 0$, such that the $i$th cohomology $H_c^i(X_r, \Lambda)[\phi]$ of $R\Gamma_c(X_r,\Lambda)[\phi]$ is non-zero if and only if $i = N_\phi$.  Moreover, in this case $H_c^{N_\phi}(X_r, \Lambda)[\phi]$ $\Lambda$-free.
\end{corollary}

This corollary verifies an important assumption in \cite[Corollary 10.4.2]{Feng_24}. We thus get a description of the Fargues--Scholze parameters of the deep level Deligne--Lusztig representations.

\begin{corollary}\label{cor:FS-param}
Let $T,U$ be as in \Cref{thm:RGamma}. Assume $\Lambda = \lambda$ and let $\phi \colon T(k) \to \lambda^\times$ be a toral character, that is, $G^0 = T$. Then
\[
\pi_{T,U, \phi} := {\rm c}\text{-}{\rm ind}_{Z(k)\CG_{\bf x}(\CO_k)}^{G(k)} H_c^{N_\phi}(X_r,\lambda)[\phi],
\] 
is an irreducible supercuspidal representation of $G(k)$, whose Fargues--Scholze parameter is 
\[ W_k \stackrel{{}^L\phi}{\longrightarrow} {}^L T(\lambda) \stackrel{{}^L\phi}{\longrightarrow} {}^L G(\lambda), \]
where ${}^L\phi$ is the L-parameter given by class field theory and ${}^L j$ is the canonical L-embedding (notation as in \cite[Theorem 10.4.1]{Feng_24}).
\end{corollary}

\begin{proof} This follows from \cite[Corollary 10.4.2]{Feng_24}, \Cref{cor:concentration_intro} and \Cref{prop:irred_modular}. \end{proof}




Moreover, Theorem \ref{thm:RGamma} can be regarded as a stronger form of \cite[Conjecture 6.5]{ChanO_21}, which follows as a special case. For any (twisted) rational Levi subgroup $T \subseteq L \subseteq G$, there is a closed subvariety $X_r^{(L)} \subseteq X_r$, called a Drinfeld stratum, see \cite[\S3]{CI_DrinfeldStrat}, \cite[\S6.2]{ChanO_21}.

\begin{corollary}\label{conj:CO_conj} Let $T \subseteq L \subseteq G$ be a twisted rational Levi subgroup. If $\phi$ is such that $(\CG^0)_{\bfx}  \otimes_{\CO_k} \BF_q \subseteq L_{\bfx} \otimes_{\CO_k} \BF_q$, then $R\Gamma_c(X_r,\Lambda)[\phi] = R\Gamma_c(X_r^{(L)},\Lambda)[\phi]$. With other words, \cite[Conjecture 6.5]{ChanO_21} holds true.
\end{corollary}
\begin{proof} This follows from \Cref{vanish} by proper base change.\end{proof}

\subsection*{Pro-unipotent Deligne--Lusztig varieties}

In the second part of the article we prove \cite[Conjecture 1.2]{IvanovNie_24}, thereby generalizing the \cite[Theorem 1.1]{IvanovNie_24}. Let $\CG^+_{\bf x}$ denote the pro-unipotent radical of $\CG_{\bf x}$ and let $\CT^+_{\bf x}$ be the closure of $T$ in $\CG^+_{\bf x}$. Very similar to $X_r$, one can define a scheme $X^+$ over $\ov\BF_q$ and its truncations $X_r^+$ (such that $X^+ = \prolim_r X_r^+$), equipped with natural $\CG_{\bf x}^+(\CO_k) \times \CT_{\bf x}^+(\CO_k)$-actions. See \S\ref{sec:prounipotent} for precise definition.

In \emph{loc.~cit.} we gave an essentially complete description of the homology of $X^+$ as a $(\BG^+)^F \times (\BT^+)^F$-module under some mild restrictions on $p$ and the condition that $T \subseteq G$ is of Coxeter type. Here, we generalize this in three ways:
\begin{itemize}
 \item[(i)] we prove the result for all elliptic tori $T$,
 \item[(ii)] we relax the assumptions on $p$ (we only require $p$ to not be a torsion prime for $G$),
 \item[(iii)] we allow modular coefficients $\Lambda$.
\end{itemize}
Towards (iii), note in particular that $\CG_{\bf x}^{+}(\CO_k)$, $\CT_{\bf x}^{+}(\CO_k)$ are pro-$p$-groups, so their representation theories in $\Lambda$-modules are semisimple (but note that for $\Lambda = \CO_L$ we still need to take the derived $\chi$-weight part as a priori there might be $\Lambda$-torsion; ultimately, we see that no torsion occurs). As above we denote the (derived) $\chi$-weight part by $(-)[\chi]$. The only subtle point arising due to modular coefficients is handled in \Cref{prop:BW_modular}, following \cite[Lemma 2.3]{ImaiT_23}.
As in \cite{IvanovNie_24}, we phrase our result in terms of the homology $f_\natural \Lambda$ of the structure map $f \colon X^+ \to \Spec \BF_q$ (whose $\chi$-weight part agrees up to a shift with the $\chi$-weight part of the compact support cohomology of $X_r^+$ for sufficiently big $r$). We refer to \cite[\S2.7]{IvanovNie_24} for a brief discussion of the homology functor. Let $N$ denote the order of $F$ as an automorphism of the root system $\Phi$ of $G$. The following generalizes \cite[Theorem 1.1]{IvanovNie_24} and proves \cite[Conjecture 1.2]{IvanovNie_24}, except that in part (3) we have to assume convexity and in part (2) a different sign might appear.

\begin{theorem}\label{thm:generalization_first_article} Assume that $T$ is elliptic. Let $\chi \colon \CT_{\bf x}^{+}(\CO_k) \to \Lambda^\times$ be a smooth character. Then the following hold.
\begin{itemize}
\item[(1)] Assume that $p$ is not a torsion prime for $G$. The homology  $f_\natural \Lambda[\chi]$ is non-vanishing in precisely one degree $s_{\chi} \geq 0$. Moreover, $H_{s_\chi}(X^+,\Lambda)[\chi] := H^{-s_\chi}(f_\natural\Lambda[\chi])$ is a free $\Lambda$-module. 

\smallskip

\item[(2)] Assume that $p$ is not a torsion prime for $G$. The Frobenius $F^N$ acts in $H_{s_\chi}(X^+,\Lambda)[\chi]$ as multiplication by the scalar $(-1)^{s_\chi'} q^{s_\chi N/2}$ with some $s_{\chi}' \in \BZ$. In particular, all Moy--Prasad quotients of $X^+$ are $\BF_{q^N}$-maximal varieties.

\smallskip

\item[(3)] Assume that the element $w\sigma \in W\sigma$ attached to $F$ in \S\ref{notation:4} is convex. For varying $\chi$, $H_{s_\chi}(X^+,\Lambda)[\chi]$ runs through pairwise non-isomorphic irreducible smooth $\CG^{+}(\CO_k)$-representations.
%
\end{itemize}
\end{theorem}

First, we remark that for part (3) the same proof as in \cite[\S7.1]{IvanovNie_24} applies, as for convex elements the (twisted) Steinberg cross-section map is an isomorphism by \Cref{thm:convex_elements}(2). It remains to prove parts (1) and (2) of \Cref{thm:generalization_first_article}. We do this in \S\ref{sec:prounipotent} by following the strategy of \cite[\S5]{IvanovNie_24}.



Let us finally give an application to some trace formulae. 

\begin{corollary}\label{cor:concentration}
Assume $\Lambda = \cool$. Let $T,U$ be as in \Cref{thm:RGamma}. Assume that $(\CG^0)_{\bfx} \otimes_{\CO_k} \BF_q = \CT_{\bfx}$ and that $p$ is not a torsion prime for $G$. Then the $R\Gamma_c(X_r,\Lambda)[\phi]$ is irreducible,  concentrated in a single degree $s_{\phi,r} \in \BZ$ and $F^N$ acts in $H_c^{s_{\phi,r}}(X_r,\Lambda)$ by the scalar $(-1)^{s_{\phi,r}'} q^{Ns_{\phi,r}/2} \in \Lambda^\times$ for some $s_{\phi,r}' \in \BZ$. Moreover, if $g \in \CG_{\bf x}(\CO_k)$, then
\[
\tr(g, H^{s_{\phi,r}}_c(X_r,\cool)[\phi]) = \frac{(-1)^{s_{\phi,r} - s_{\phi,r}'}}{|\BT_r^F| \cdot q^{Ns_{\phi,r}/2}} \sum_{t \in \BT_r^F} \phi(t) \cdot |S_{g,t}|,
\]
where $S_{g,t} = \{x \in X_r(\ov\BF_q) \colon gF^N(x) = xt\}$.
\end{corollary}
\begin{proof} Indeed, by \Cref{conj:CO_conj}, $R\Gamma_c(X_r,\Lambda)[\phi] = R\Gamma_c(X^{(T)}_r,\Lambda)[\phi]$. But $X_r^{(T)}$ is a disjoint union (indexed over $\BG_0^F$) of copies of the scheme $X^+_r$ from \S\ref{sec:prounipotent} and by \Cref{thm:generalization_first_article}, $R\Gamma_c(X^+_r,\Lambda)[\phi]$ is concentrated in one degree. This implies concentration in one degree. The assumption $G^0=T$ and \cite[Lemma 3.6.5]{Kaletha_19} imply ${\rm Stab}_{W^F}(\phi) = 1$. Then $R\Gamma_c(X_r,\Lambda)[\phi]$ is irreducible by (for example) \cite[Theorem 1.6]{Nie_24}. The last claim follows by applying \cite[Lemma 2.12]{Boyarchenko_12}.
\end{proof}

\subsection*{Acknowledgements} The first author gratefully acknowledges the support of the German Research Foundation (DFG) via the Heisenberg program (grant nr. 462505253). He would like thank Jessica Fintzen and David Schwein for several clarifying explanations (in particular, for explaining Lemma \ref{lm:closed_addition} to him).

\smallskip

\section{Notation and setup}

\subsection{General notation}\label{notation:1} We let $k \subseteq \breve k$ with integer rings $\CO_k \subseteq \CO$, residue field extensions $\BF_q \subseteq \obF$, and Frobenius $F$ be as in the introduction. We denote by $\varpi$ a uniformizer of $k$.

For a perfect $\BF_q$-algebra $R$, put $\BW(R) = R[\![\varpi]\!]$ if ${\rm char}(k) > 0$, resp. $\BW(R) = W(R) \otimes_{\BZ_p} \CO_k$ if ${\rm char}(k) = 0$, where $W(R)$ denotes the ring of Witt vectors of $R$. In particular, we have $\BW(\BF_q) = \CO_k$ and $\BW(\obF) = \CO$. Let $[\cdot] \colon R \rar \BW(R)$ be the Teichm\"uller lift if ${\rm char}(k) = 0$, resp. $[x] = x$ if ${\rm char}(k) > 0$.

Let $\CX$ be an $\CO$-scheme, which is affine and of finite type over $\CO$. Applying the (perfect) positive loop functor $L^+$ to $\CX$ yields a perfect affine $\obF$-scheme
\[
\BX = L^+\CX \quad \text{satisfying} \quad \BX(R) = \CX(\BW(R))
\]
for any perfect $\obF$-algebra $R$. If $\CX$ is defined over $\CO_k$, then $\BX$ is naturally defined over $\BF_q$, and we denote by $F$ the (geometric) Frobenius acting on $\BX(\obF)$, so that $\BX^F = \BX(\BF_q)  = \CX(\CO_k)$.

We let the prime $\ell \neq p$ and the coefficient ring $\Lambda$ be as in the introduction. The $6$-functor formalism of \'etale cohomology with coefficients in $\Lambda$ attaches to any variety $X$ over $\overline\BF_q$ the complex $R\Gamma_c(X,\Lambda) \in D^b(\Lambda{-\rm mod})$ of cohomology with compact support. If $X$ is equipped with the action of a (finite) group $H$, then we get a complex $R\Gamma_c(X,\Lambda) \in D^b(\Lambda H{-\rm mod})$.

If $C \in D^b(\Lambda{-\rm mod})$ and $\phi$ is an irreducible $H$-representation, then we write $C[\phi] := C \otimes_{\Lambda H}^L \phi$ for the derived $\phi$-isotypic part of $C$. If $\Lambda H$ is semisimple, or more generally, if $C$ is represented by a complex of projective $\Lambda H$-modules, then $C[\phi]$ is equal to the usual $\phi$-isotypic part $C \otimes_{\Lambda H} \phi$.

We denote the cardinality of a finite set $X$ by $|X|$.

\subsection{Groups}\label{notation:2} We fix a reductive group $G$ defined over $k$ and split over $\breve k$. We write $Z = Z(G)$ for the center of $G$, $G_{\rm der}$ for the derived group of $G$, and $G_{\rm sc}$ for the simply connected cover of $G_{\rm der}$. 

Let ${\bf x}$ be a point of the (reduced) Bruhat--Tits building of $G$ over $k$. By Bruhat--Tits theory there is an associated connected parahoric $\CO_k$-model $\CG_{\bf x}$ of $G$, equipped with filtration by the Moy--Prasad subgroups $\CG_{\bf x}^r$ for $r \in \BR_{\geq 0}$ ($\CG_{\bf x}^r(\CO)$ contains exactly the subgroups attached to affine roots $f$ with $f({\bf x})\geq r$). We let
\[
J = {\rm Jumps}({\bf x},G) = \{r \in \BR_{\geq 0} \colon \CG_{\bf x}^r \neq \CG_{\bf x}^{r'} \text{ for all $r'>r$} \},
\]
This is a discrete subset of $\BR_{\geq 0}$, and for $r \in \BR_{\ge 0}$ we denote by $r+ = \min \{s \in J; r < s\} \in J$ and $r- = \max \{s \in J; s < r\}$ (if it exists) its descendant and ascendant respectively.

For any $s \leq r \in \BR_{\ge 0}$ we obtain the $\BF_q$-rational perfectly smooth affine (Moy--Prasad) group scheme
\[ \BG_r^s := \BG^s / \BG^{r+}, \]
where $\BG^s = \BG_{\bf x}^s = L^+ \CG_{\bf x}^s$. We will write $\BG = \BG^s$ and $\BG_r = \BG_r^s$ if $s = 0$. Note also that $\BG_r^s(\BF_q)$ is a finite Moy--Prasad subquotient of the $p$-adic reductive group $G(k)$.

Let $H \subseteq G$ be closed subgroup defined over $\breve k$. We may consider its closure $\CH$ in $\CG$, apply $L^+$ and pass to (sub)quotients to obtain a closed subgroup $\BH_r^s \subseteq \BG_r^s$ over $\ov \BF_q$ (see \cite[\S2.6]{CI_MPDL}). If $H$ was $k$-rational, then $\BH_r^s$ is $F$-stable. By abuse of notation we will identify $H = H(\brk)$.

\subsection{Pinning}\label{notation:3} We fix an elliptic $k$-rational, $\breve k$-split maximal torus $T$ of $G$, and we denote by $N_G(T)$ its normalizer. We identify its Weyl group $W= N_G(T)/T$ with the set of its $\breve k$-points; it is endowed with a natural action of $F$. We denote by $X_\ast(T)$, $X^\ast(T)$ the groups of (co)characters of $T_{\breve k}$, equipped with natural $F$-actions, and by $\langle,\rangle \colon X^\ast(T) \times X_\ast(T) \rar \BZ$ the natural $W$- and $F$-equivariant pairing. We will write $T_{\rm der},T_{\rm sc}$ for the preimage of $T$ in $G_{\rm der},G_{\rm sc}$, respectively.

We fix a Borel subgroup $T \subseteq B \subseteq G$ defined over $\breve k$, we denote by $U$ the unipotent radical of $B$, and by $\overline U$ the unipotent radical of the opposite Borel subgroup. We write $\Phi = \Phi_G \subseteq X^\ast(T)$ for the set of roots of $T$ in $G$, and $\Phi^+$ resp. $\Phi^-$ for the subset of positive roots corresponding to $U$ resp. $\overline U$. For each $\alpha \in \Phi$, let $U_\alpha \colon \BG_{a,\breve k} \to G$ denote a fixed parametrization of the root subgroup of $\a$. For $V \subseteq G(\breve k)$, we write $\Phi_V = \{\alpha \in \Phi \colon U_\a(\breve k) \subseteq V \}$.

\subsection{Factorization of Frobenius}\label{notation:4} There is a unique element $w \in W$, such that $FB = {}^wB$. Moreover, for any lift $\dot w \in N_G(T)(\breve k)$, $\Ad (\dot w)^{-1} \circ F \colon G(\breve k) \to G(\breve k)$ fixes the pinning $(T,B)$ of $G$, and hence defines an automorphism $\sigma$ of the Coxeter system $(W,S)$, where $S$ is the set of simple reflections determined by $B$. Moreover, there is a unique automorphism of $X^\ast(T)$, again denoted by $\sigma$, such that the $F$-action on $X^\ast(T)$ is given by $qw\sigma$. This defines an action of $W \rtimes \langle \sigma \rangle$ on $X^\ast(T)$ satisfying $\sigma(\Phi^+) = \Phi^+$. We denote by $w\s \in W\s$ the relative position of $B$ and $FB$.


\subsection{Affine roots}\label{sec:affine_roots}
Denote by $\CT$ the connected N\'eron model of $T$. Then $\CT(\CO)$ is the maximal bounded subgroup of $T(\breve k)$.
Moreover, for $r \in \BZ_{\geq 0}$,
\[
\CT(\CO)^r = \{t \in \CT(\CO) \colon \ord_{\varpi}(\chi(t) - 1) \geq r \, \forall \chi \in X^\ast(T) \}
\]
defines a descending separated filtration on $T(\breve k)$, satisfying $\CT(\CO)^0 = \CT(\CO)$. For $r \geq 1$ one has an isomorphism
\[
V := X_\ast(T) \otimes \obF \stackrel{\sim}{\longrightarrow} \CT(\CO)^r /\CT(\CO)^{r+1}, \quad \lambda \otimes x \mapsto \lambda(1 + [x]\varpi^r).
\]
We denote by $\Phi_\aff \cong \Phi \times \BZ$ the set of affine roots of $T$ in $G$ (with respect to a fixed point in the apartment of $T$ in the Bruhat--Tits building of $G$).
For $f \in \Phi_\aff$, we write $\alpha_f \in \Phi$ for its vector part and $n_f \in \BZ$ for the integer such that $f=(\alpha_f,n_f)$.
We write $\widetilde\Phi = \Phi_{\rm aff}\, \sqcup \,\BZ_{\geq 0}$ for the enlarged set of affine roots, with the affine root subgroup corresponding to $r \in \BZ_{\geq 0}$ being the $r$-th slice of $T(\CO)$. There is a natural $F$-action on $\Phi_\aff$, and we extend it to an $F$-action on $\tPhi$ by letting $F$ act trivially on $\BZ_{\geq 0}$.

\subsection{Preparations in the modular case} \label{sec:modular_preparations}

We have the following generalization of \cite[Proposition 6.6.1]{BoyarchenkoW_16} for arbitrary coefficients $\Lambda$, due to \cite{ImaiT_23}. 

\begin{proposition}\label{prop:BW_modular} Let $Q$ be a power of $p$.  Let $\psi \colon \BF_{Q^2} \to \Lambda^\times$ be a non-trivial character which is trivial on $\BF_Q$. Let $\CL_\psi$ denote the corresponding Artin--Schreier local system on $\BG_{a,\BF_{Q^2}}$. Let $f \colon \BG_a \to \BG_a$ the map $x\mapsto x^{Q+1}$. Then we have isomorphisms of $\Lambda$-modules, 
\[
H_c^i(\BG_a,f^\ast\CL_\psi) \cong \begin{cases} \Lambda^{\oplus Q} & \text{if $i=1$} \\ 0 & \text{otherwise,} \end{cases}
\]
Moreover, the Frobenius ${\rm Fr}_{Q^2}$ acts on $H_c^1(\BG_a,f^\ast\CL_\psi)$ by multiplication with $-Q \in \Lambda^\times$.
\end{proposition}
\begin{proof} The case $\Lambda = \cool$ is \cite[Proposition 6.6.1]{BoyarchenkoW_16}. When $\Lambda = \CO_L$ or its residue field $\lambda$, we use the argument of \cite[Lemma 2.3]{ImaiT_23}, which we now sketch. Let $\psi' \colon \BF_{Q^2}/\BF_Q \cong \BF_Q \to \Lambda^\times$ be the character through which $\psi$ factors, so that $\CL_\psi = \CL_{\psi'}$, where the latter is the Artin--Schreier sheaf corresponding to $\psi'$ and the covering $y \to y^q+y \colon \BG_a \to \BG_a$. The smooth affine curve $C = \{y^q+y = x^{q+1}\}$ is the pullback of this covering along $f$. Thus we have $R\Gamma_c(C,\Lambda)[\psi] = R\Gamma_c(\BG_a,f^\ast\CL_{\psi})$. The cohomology of a smooth affine curve with coefficients in $\CO_L$ is $\CO_L$-torsion free. Thus we even have $H^i_c(C,\Lambda) \otimes_{\Lambda\BF_Q} \psi = H_c^i(\BG_a,f^\ast \CL_{\psi})$ when $\Lambda = \CO_L$. The result for $\Lambda \in \{\CO_L,\lambda\}$ now follows from the case $\Lambda = L$.
\end{proof}



We will also need the following lemma.

\begin{lemma}\label{lm:integral_coefficients}
Let $X$ be a variety over $\ov\BF_q$ endowed with the action of a finite group $H$. Let $\chi \colon H \to \CO_L^\times$ be a character and $\bar\chi \colon H \to \lambda^\times$ be its composition with $\CO_L^\times \to \lambda^\times$. 

\begin{itemize}
\item[(1)] Then $R\Gamma_c(X,\lambda)[\bar\chi] \cong R\Gamma_c(X,\CO_L)[\chi] \otimes_{\CO_L}^L \lambda$. 

\item[(2)] Assume that $|H| \in \CO_L^\times$.
If $\dim_L H^i(X,L)[\chi] = \dim_\lambda H^i(X,\lambda)[\bar\chi]$ for all $i \in \BZ$, then the cohomology of $R\Gamma_c(X,\CO_L)[\chi]$ is $\CO_L$-free.
\item[(3)] If in the situation of (2) there is an $s \in \BZ$ such that $H_c^i(X,L)[\chi] = 0$ for $i \neq s$, then the same holds for $R\Gamma_c(X,\CO_L)[\chi]$. 
\end{itemize}
\end{lemma}
\begin{proof}
For (1) we compute
\begin{align*}
R\Gamma_c(X,\lambda)[\bar\chi] &\cong (R\Gamma_c(X,\CO_L) \otimes^L_{\CO_L} \lambda) \otimes_{\lambda H}^L \bar\chi \\
&\cong (R\Gamma_c(X,\CO_L) \otimes_{\CO_L H}^L \chi) \otimes^L_{\CO_L} \lambda \\
&\cong R\Gamma_c(X,\CO_L)[\chi] \otimes_{\CO_L}^L \lambda
\end{align*}
Now assume that $|H| \in \CO_L^\times$. In particular, by \cite{Rickard_95},
$R\Gamma_c(X,\CO_L)$ is represented by a perfect complex of $\Lambda H$-modules, and in particular $R\Gamma_c(X,\CO_L)[\chi]$ is represented by a perfect complex $C$ of $\Lambda$-modules. As in the universal coefficient theorem we have short exact sequences for each $i \in \BZ$:
\begin{equation*}\label{eq:universal_coeff}
0 \to H^i(C)\otimes_{\CO_L} \lambda \to H^i(C \otimes_{\CO_L} \lambda) \to {\rm Tor}_1^{\CO_L}(H^{i+1}(C),\lambda) \to 0.
\end{equation*}
This implies the inequality
\[ H^i(X,\lambda)[\bar\chi] = \dim_\lambda H^i(C \otimes_{\CO_L} \lambda) \geq \dim_\lambda H^i(C)\otimes_{\CO_L} \lambda,\] where the equality follows from the first claim (along with the assumptions). On the other hand, 
\[\dim_\lambda H^i(C) \otimes_{\CO_L} \lambda \geq {\rm rk}_{\CO_L} H^i(C) = \dim_L H^i_c(X,L)[\chi],\] 
where equality again follows from the universal coefficient theorem.
Putting these inequalities together and using the assumption of (2), we deduce that the ${\rm Tor}$-term in the \eqref{eq:universal_coeff} vanishes. This shows (2). Now (3) follows as $H^i(C)$ is $\CO_L$-torsion-free by (2) and $H^i(C) \otimes L \cong H^i(X,L)[\chi]$.
\end{proof}

\section{Convex elements}\label{sec:convex}
In this section we introduce elliptic convex elements in $W \rtimes \<\s\>$, where $\s$ is the automorphism of $(W, S)$ as in \Cref{notation:4}. They behave like Coxeter elements in many respects, but they have the advantage that any elliptic conjugacy class of  $W \rtimes \<\s\>$ contains a convex element, as is proven by the Tan, Yu and the second author in \cite{NieTanYu_24}. In later sections we will make use of the fact that deep level Deligne--Lusztig varieties attached to convex elements of $W \rtimes \<\s\>$ can be studied by similar techniques as in the Coxeter case.


\subsection{Elliptic convex elements }

Let $x \in W \rtimes \<\s\>$ be an elliptic element, that is, there is no nonzero element in $\BR \Phi$ fixed by $x$. Set
\[
\D_x = \Phi^+ \cap x(\Phi^-).
\]
For $\a \in \Phi^\pm$ we set \[n_x(\a) = \min\{i \in \BZ_{\ge 1}; x^i(\a) \in \Phi^\mp\} \in \BZ_{\ge 1},\] which is well-defined since $x$ is elliptic.

\begin{definition}
We say an elliptic element $x \in W\sigma$ is \emph{quasi-convex} if \[n_x(\a + \b) \le \max\{n_x(\a), n_x(\b)\}\] for all $\a, \b \in \Phi^\pm$ such that $\a + \b \in \Phi$. Moreover, we say  $x$ is \emph{convex} if both $x$ and $x\i$ are quasi-convex.
\end{definition}

\begin{lemma}
    Let $x$ be a elliptic quasi-convex element. Let $\a, \b \in \Phi^+$ and $i, j \in \BZ_{\ge 1}$ such that $i\a + j\b \in \Phi^+$. Then $n_x(i\a + j\b)  \le \max\{n_x(\a), n_x(\b)\}$.
\end{lemma}
\begin{proof}
We always can find a sequence of roots $\gamma_0,\gamma_1,\dots,\gamma_t = i\a + j\b$ with $\gamma_0$ and $\gamma_m - \gamma_{m-1}$ ($\forall 1 \leq m \leq t$) equal either $\a$ or $\b$. Then the lemma follows by induction from the definition.
\end{proof}

Convex elements were studied in \cite{NieTanYu_24}, where the following was proven.

\begin{theorem}[\cite{NieTanYu_24}, Theorem 0.1]\label{thm:convex_elements} The following statements hold true.

(1) In each elliptic $W$-conjugacy class of $W \s$, there exists a convex element.

(2) (Steinberg cross-sections) For any elliptic convex element $x \in W \rtimes \<\s\>$ we have an isomorphism  \[({}^x U \cap U) \times (\ov U \cap {}^x U) \to {}^x U, \quad (h, g) \mapsto h^{-1} g \Psi(h),\] where $\Psi: {}^x U \cap U \to {}^x ({}^x U \cap U)$ is any isomorphism sending the root subgroup of $\a$ to the root subgroup of $x(\a)$ for all $\a \in x(\Phi^+) \cap \Phi^+$. 
\end{theorem}

We will need further properties of convex elements.

\begin{lemma} \label{order}
    Let $x \in W \s$ be convex. Let $\a, \b \in \Phi$ such that $\b - \a \in \BZ_{\ge 0} \D_x$. Then

    (1) if $\a \in \Phi^+$ then $n_{x\i}(\b) \le n_{x\i}(\a)$;

    (2) If $\a, x\i(\a) \in \Phi^-$ then either $\b \in \Phi^+$ or $n_x(\b) \le n_x(\a)$.
\end{lemma}
\begin{proof}
    (1) By assumption, there exists a sequence of roots \[\a = \g_0, \g_1, \dots, \g_m = \b\] such that $\g_i - \g_{i-1} \in \D_x$. Since $x\i$ is quasi-convex and $n_{x\i}(\D_x) = \{1\}$ we deduce that \[n_{x\i}(\a) = n_{x\i}(\g_0) \ge n_{x\i}(\g_1) \ge \cdots \ge n_{x\i}(\g_m) = n_{x\i}(\b)\] as desired.

    (2) We can assume that $\a, x\i(\a), \b \in \Phi^-$. Since $x\i(\D_x) \in \Phi^-$, we have  $x\i(\b) \in \Phi^-$ and $n_x(x\i(\D_x)) = \{1\}$. Note that \[x\i(\b) - x\i(\a) \in \BZ_{\ge 0} x\i(\D_x) = - \BZ_{\ge 0} \D_{x\i}.\] Thus, by (1) we have \[n_x(\b) + 1 = n_x(x\i(\b)) \le n_x(x\i(\a)) = n_x(\a) + 1,\] which implies that $n_x(\b) \le n_x(\a)$ as desired.
\end{proof}

\subsection{$M$-standard convex elements}
Let $M \subseteq G$ be an $F$-stable Levi subgroup containing fixed maximal torus $T$. We denote by $W_M \subseteq W$ the Weyl of $M$.
\begin{proposition} \label{standard}
    There exists a Borel subgroup $B \supseteq T$ such that 
    
    (1) $M$ is a standard Levi subgroup with respect to $B$;

    (2) the relative position $x \in W\s$ of $B$ and $FB$ (see \Cref{notation:4}) is a convex element with respect to the Coxeter system $(W, S)$ attached to $B$.
\end{proposition}
\begin{proof}
    Let $V = \BR\Phi$ be the Euclidean space together with an inner product $( , )$ preserved by $W \rtimes \<\s\>$. Since $T$ is elliptic, there exists an orthogonal decomposition \[V = \oplus_{i=1}^n V_i,\] where each $V_i$ is an $F$-stable subspace of dimension $\le 2$. Moreover, for each $i$ there exist $0< \th_i \le \pi$ such that $F(v) + F\i(v) = 2\cos\th_i \cdot v$ for all $v \in V_i$.
    
    Let $V_M \subseteq V$ be the subspace spanned by the roots of $M$. Denote by $V_M^\perp$ be the orthogonal complement of $V_M$. As $M$ is $F$-stable, $V_M^\perp$ is preserved by $x$. By reordering the subspaces $V_i$, we may assume that $V_M^\perp = \oplus_{i=1}^m V_i$ for some $0 \le m \le n$. By \cite[Lemma 5.1]{He-Nie_12}, there exists a Weyl chamber $C \subseteq V$ for $\Phi$ such that for each $1 \le i \le n$ its Hausdorff closure $\ov C$ contains a $\Phi$-regular point of $\oplus_{j=1}^i V_j$. Here for any linear subspace $V' \subseteq V$ a point $v' \in V'$ is called a $\Phi$-regular point of $V'$ if for each $\a \in \Phi$, $(\a, v') = 0$ implies that $(\a, V') = \{0\}$. 

    Let $T \subseteq B$ be the Borel subgroup associated to the Weyl chamber $C$. As $\ov C$ contains a regular point $V_M^\perp = \oplus_{i=1}^m V_i$, $M$ is a standard Levi subgroup with respect to $B$. Moreover, by \cite[Theorem 3.4]{NieTanYu_24}, the relative position of $B$ and $FB$ is a convex element with respect to the Coxeter system $(W, S)$ attached to $B$. The proof is finished.
\end{proof}

\subsection{Action of convex elements on a Lie algebra}\label{sec:Lie_alg} Let $x \in W \rtimes \<\s\>$. For $A \subseteq \Phi$ we consider the following $\ov\BF_q$-vector spaces
\[
H_A = \bigoplus_{\a \in A} \ov\BF_q e_\a \subseteq H_\Phi = \bigoplus_{\a \in \Phi} \ov\BF_q e_\a.
\] Assume that $A = x(A)$. Then we denote by $F = F_A$ the Frobenius map on $H_A$ given by $F(c e_\a) = c^q e_{x(\a)}$ for $c \in \ov\BF_q$.

Let $B \subseteq -\D_x = \Phi^- \cap x(\Phi^+)$ such that for any $\a \in A$, $\b \in B$ and $i \in \BZ_{\ge 1}$ we have $\a + i\b \in A$ if $\a + i\b \in \Phi$. For $\b \in B$ and $c \in \ov\BF_q$ we define a linear map
\[
\Ad_\b(c): H_A \to H_A, \quad e_\a \mapsto e_\a + \sum_{\substack{i \ge 1: \\ \a + i\b \in \Phi}} c_{\a, \b, i} c^i e_{\a + i\b},
\]
where $c_{\a, \b, i} \in \ov\BF_q$ are arbitrary but fixed constants.

Assume $B = \{\b_1, \dots, \b_n\}$. Let $\phi = \Ad_{\b_1}(c_1) \circ \cdots \circ \Ad_{\b_n}(c_n)$, where $c_j \in \ov\BF_q$ for $1 \le j \le n$ are arbitrary but fixed. For a fixed $z \in H_A$, let
\[
V(\phi, x, z) := \{w \in H_A; \phi(w) - F(w) - z \in H_{A \cap -\D_x}\}.
\]
This is a closed subvariety of $H_A$. In \S \ref{sec:fibers_over_Y0} we will use it to describe the fibers of a deep level Deligne--Lusztig variety over one of a shallower depth.
Now we prove the following general structure result for $V(\phi,x,z)$.

\begin{proposition} \label{uniformization}
    Let notation be as above. Assume that $x$ is convex. Then the natural projection $H_A \to H_{A \cap \D_x}$ induces a homeomorphism \[V(\phi, x, z) \cong H_{A \cap \D_x}.\]
\end{proposition}
\begin{proof}
    Write $w = \sum_{\a \in A} w_\a e_\a$, $z = \sum_{\a \in A} z_\a e_\a$ and $\phi(w) = \sum_{\a \in A} y_\a e_\a$ with $w_\a,z_\a,y_\a \in \ov\BF_q$. Then the variety $V(\phi, x, z)$ is defined by the equations \begin{align*}
        \tag{$E_\a$} y_\a - w_{x\i(\a)}^q - z_\a = 0,
    \end{align*} where $\a$ ranges over the roots in $A \sm (-\D_x)$.

    For $\a \in A$, we set $\G_\a = (\a + \BZ_{\ge 0} B)  \cap A$. As $B \subseteq -\D_x$, it follows from the definition of $\phi$ that \[y_\a \in \sum_{\g \in \G_\a} c_\a^\g w_\g,\] where $c_\a^\g \in \overline \BF_q$ are some constants such that $c_\a^\a = 1$. Hence the equation $(E_\a)$ is equivalent to
    \begin{align*}
        \tag{$E_\a'$} w_\a - w_{x\i(\a)}^q = z_\a - \sum_{\g \in \G_\a \sm \{\a\}} c_\a^\g w_\g.
    \end{align*}

    Now we show that given $z$ and $(w_\a)_{\a \in A \cap \D_x}$ there exists a unique tuple $(w_\a)_{\a \in A \sm \D_x}$ such that the equations $(E_\a')$ hold for all $\a \in A \sm (-\D_x)$. To this end, for $\a \neq \b \in \Phi^+$ we define $\b \prec \a$ if either $n_{x\i}(\b) < n_{x\i}(\a)$ or $n_{x\i}(\b) = n_{x\i}(\a)$ and $\b - \a$ is a sum of roots in $\D_x$.

    First we claim that $w_\a$ is determined by the equation $(E_\a')$ for $\a \in (A \cap \Phi^+) \sm \D_x$ (by which we mean that we may eliminate equation $(E_\a')$ along with the variable $w_\a$). We use induction on the partial order $\preceq$ on $A \cap \Phi^+$. As $\a \in (A \cap \Phi^+) \sm \D_x$, we have $x\i(\a) \in \Phi^+$ and hence $x\i(\a) \prec \a$. Moreover, by \Cref{order} (1) we have $\g \prec \a$ for $\g \in \G_\a \sm \{\a\}$. By induction hypothesis, $w_{x\i(\a)}$ and $w_\g$ for $\g \in \G_\a \sm \{\a\}$ are already determined. Hence $w_\a$ is determined by the equation $E_\a'$, and the claim is proved.

    It remains to show that $w_\a$ is determined by the equation $(E_{x(\a)}')$ for $\a \in A \cap \Phi^-$. We argue by induction on $n_x(\a)$. In view of $(E_{x(\a)}')$, $w_\a$ is determined by $z_{x(\a)}$ and $w_\g$ for $\g \in \G_{x(\a)}$. So it suffices to show $w_\g$ is already determined for $\g \in \G_{x(\a)}$. Indeed, if $\g \in \Phi^+$, this follows from the previous claim. Now we assume $\g \in \Phi^-$ and hence $x(\a) \in \Phi^-$. By Lemma \ref{order} (2), we have $n_x(\g) \le n_x(x(\a)) < n_x(\a)$. Thus $w_\g$ is determined by the induction hypothesis. Thus $w_\a$ is determined by the equation $(E_{x(\a)}')$, and the proof is finished.
\end{proof}

\begin{proposition} \label{Steinberg}
    Let notation be as above. Assume that $x$ is convex. Then the map $(z, y) \mapsto -\phi(z) + y - F(z)$ gives an isomorphism \[H_{A \cap x(\Phi^+) \cap \Phi^+} \times H_{A \cap -\D_x} \overset \sim \to H_{x(A \cap \Phi^+)}.\]
\end{proposition}
\begin{proof}
By \Cref{uniformization}, this map is injective. It suffices to show it is surjective, that is, for any $z \in H_{x(A \cap \Phi^+)}$, there exists $w \in H_{A \cap x(\Phi^+) \cap \Phi^+}$ such that $\phi(w) - F(w) - z \in H_{x(A \cap -\D_x)}$. This is equivalent to the following statement:

(a) For any $z \in H_{A \cap \Phi^+}$, there exists $w \in H_{A \cap x(\Phi^+) \cap \Phi^+}$ such that $-\varphi(w) + F\i(w) - z \in H_{A \cap \D_{x\i}}$. Here $\varphi = F\i \circ \phi \circ F = \prod_{\g \in x\i(B) \cap \D_{x\i}} \Ad_\g(d_\g)$ for some $d_\g \in \ov\BF_q$.

Now we prove (a). Let $z = \sum_\a c_\a e_\a \in H_{A \cap \Phi^+}$ for some  $c_\a \in \ov\BF_q$. Define \[n_x(z) = \max\{n_x(\a); c_\a \neq 0\}.\] We argue by induction on $n_x(z)$. If $n_x(z) = 1$, that is, $z \in H_{A \cap \D_{x\i}}$ and we may take $w = 0$. Assume $n_x(z) \ge 2$. Let $z' = \sum_{\g, n_x(\g) = n_x(z)} c_\g e_\g \in H_{A \cap \Phi^+}$. Then $n_x(z - z') \le n_x(z) - 1$ and $F(z') \in H_{A \cap x(\Phi^+) \cap \Phi^+}$. Moreover, as $x$ is convex and $n_x(\g) = 1$ for $\g \in \D_{x\i}$, we have \[n_x(\varphi(F(z'))) \le n_x(F(z')) = n_x(z') - 1 = n_x(z) - 1.\] Thus \[n_x(\varphi(F(z')) - F\i(F(z')) - z) = n_x(\varphi(F(z')) - (z - z')) \le n_x(z) - 1.\] Then the statement follows by induction hypothesis. The proof is finished. \qedhere
\end{proof}

\section{Deligne--Lusztig varieties}\label{sec:DeligneLusztig}

\subsection{Deligne--Lusztig varieties}\label{sec:DLV}
Recall the notation from \S\ref{notation:2}. Fix $r \in J$. We have the $\ov \BF_q$-group $\BG_r$ equipped with $\BF_q$-Frobenius $F$ and its subgroups $\BT_r,\BU_r,\ov\BU_r$. Consider the $\BF_q$-varieties
\begin{align*}
X_r &= \{g \in \BG_r \colon g^{-1}F(g) \in \ov\BU_r \cap F\BU_r \} \\
Y_r &= \{ g \in \BG_r \colon g^{-1}F(g) \in \BT_r (\overline \BU_r \cap F\BU_r) \}/\BT_r.
\end{align*}
There is an obvious map $h \colon X_r \to Y_r$, which is an \'etale $\BT_r^F$-torsor with $\BT_r^F$ acting by right multiplication. Let $\CE_\theta$ denote the one-dimensional local system on $Y_r$ corresponding to a character $\theta \colon \BT_r^F \to \Lambda^\times$.

\begin{lemma}\label{lm:derived_weight_part}
Let $\th \colon \BT_r^F \to \Lambda^\times$ be a character. We have
\[
R\Gamma_c(X_r,\Lambda)[\th] = R\Gamma_c(Y_r,\CE_\theta).
\]
\end{lemma}
\begin{proof}
We compute in $D^b(\Lambda\BG_r^F{-\rm mod})$:
\begin{align*}
R\Gamma_c(X_r,\Lambda)[\th] &= R\Gamma_c(X_r,\Lambda) \otimes_{\Lambda \BT^F}^L \theta \\
&\cong R\Gamma_c(Y_r,h_\ast\Lambda) \otimes_{\Lambda \BT^F}^L \theta \\
&\cong R\Gamma_c(Y_r,h_\ast \Lambda \otimes_{\Lambda \BT_r^F}^L \theta) \\
&\cong R\Gamma_c(Y_r,h_\ast \Lambda \otimes_{\Lambda \BT^F} \CE_\theta) \\
&\cong R\Gamma_c(Y_r,\CE_\theta)
\end{align*}
Here, the second equality is by finiteness of $h$; the third follows from the projection formula applied to the map $Y_r \to [\ast/\BT_r^F]$ determined by $h$, the sheaf $h_!\Lambda = h_\ast\Lambda$ on $Y_r$ and the sheaf $\theta$ on $[\ast/\BT_r^F]$; the fourth holds because the action of $\BT_r^F$ is free, and hence the fibers of $h_\ast\Lambda$ are projective $\Lambda\BT^F$-modules (see \cite[Lemma 3.2]{BonnafeR_03}); finally, the fifth equality follows from the natural map $h_\ast\Lambda \otimes_{\Lambda \BT_r^F} \CE_\theta \to \CE_\theta$ and a stalkwise comparison (or by identifying these sheaves with $\Lambda\BT_r^F$-modules and noting that $h_\ast\Lambda$ corresponds to the free module $\Lambda\BT_r^F$).
\end{proof}

\subsection{Howe strata}\label{sec:Howe}
Fix a character $\phi: T^F = T(k) \to \Lambda^\times$ of ${\rm depth}(\phi) \le r$. Then $\phi$ induces a character $\BT^F \to \Lambda^\times$, which we again denote by $\phi$. Assume that $\phi$ admits a \emph{Howe factorization} in the sense of \cite[\S3.6]{Kaletha_19} and denote it by $(G^i, \phi_i, r_i)_{-1 \le i \le d}$. That is,
\[
T = G^{-1} \subseteq L:=G^0 \subsetneq G^1 \subsetneq \dots G^{d-1} \subsetneq G^d
\]
is a sequence of twisted Levi subgroups, $\phi_i$ ($0\leq i \leq d$) is a character of $(G^i)^F$, which is $(G^i : G^{i+1})$-generic for $i<d$, such that $\phi = \prod_{i=-1}^d \phi_i$. Moreover, there is a sequence $0 = r_{-1} <r_0 < \dots <r_{d-1} \leq r_d$ of integers such that $\phi_i$ has depth $r_i$ for $0\leq i \leq d-1$; $\phi_d = 1$ if $r_{d-1}=r_d$ and $\phi_d$ has depth $r_d$ otherwise; $\phi_{-1} = 1$ if $G^0 = T$ and $\phi_{-1}$ has depth $0$ otherwise. For $\a \in \Phi$ we denote by $i(\a)$ the unique integer $0 \le i \le d$ such that  $\a \in \Phi(G^i, T) \setminus \Phi(G^{i-1}, T)$. Define $r(\a) = r_{i(\a)-1}$.

We define subgroups of $\BG$ as follows.
\begin{align*} \BK_\phi &= (\BG^0)^0 (\BG^1)^{r_0/2} \cdots (\BG^d)^{r_{d-1}/2}; \\  \BK^+_\phi &= (\BG^0)^{0+} (\BG^1)^{r_0/2 +} \cdots (\BG^d)^{r_{d-1}/2+}; \\
\BH_\phi &= (\BG^0)^{0+} (\BG^1)^{r_0/2} \cdots (\BG^d)^{r_{d-1}/2}; \\
\BE_\phi &= (\BG^0_\der)^{0+} (\BG^1_\der)^{r_0/2+, r_0+} \cdots (\BG^d_\der)^{r_{d-1}/2+, r_{d-1}+}.
\end{align*}
Here $(\BG^i_\der)^{r_{i-1}/2+, r_{i-1}+}$ is generated by $(\BG^i_\der)^{r_{i-1}+}$ and $\BU_f$ for $f \in \tPhi_\aff^{r_{i-1}/2+}$ such that $\a_f \in R_i \sm R_{i-1}$.

Furthermore, we let $\BK_{\phi, r}$, $\BH_{\phi, r}$, $\BE_{\phi, r}$, ... be the natural images of $\BK_{\phi}$, $\BH_\phi$, $\BE_\phi$, ... in $\BG_r$ respectively.

The ``discrete part'' $(G^i, r_i)_{-1 \le i \le d}$ of the Howe datum of $\phi$ cuts out the following subvarieties of $X_r,Y_r$, which might therefore be called (closed) \emph{Howe strata} of $X,Y$:
\begin{align*} X_r^\flat &= \{g \in \BG_r \colon g\i F(g) \in \BK_{\phi, r} \cap \ov \BU_r \cap F\BU_r \},\\
Y_r^\flat &= \{g \in \BG_r \colon g\i F(g) \in \BT_r (\BK_{\phi, r} \cap \ov \BU_r \cap F\BU_r)\}/ \BT_r
\end{align*}

The following is our first main result. It says the $\phi$-isotypic part of the cohomology of $X_r$ concentrates on the corresponding Howe stratum.

\begin{theorem}\label{thm:concentrates_Howe}
Suppose the element $w\sigma \in W\sigma$ attached to $F$ in \S\ref{notation:4} is convex. We have $R\Gamma_c(X_r \sm X_r^\flat,\Lambda)[\phi] = R\Gamma_c(Y_r \sm Y_r^\flat, \CE_\phi) = 0$.
\end{theorem}

The first equality is Lemma \ref{lm:derived_weight_part}. By proper base change, \Cref{thm:concentrates_Howe} follows from the vanishing of the cohomology of $\CE_\phi$ on the fibers of $Y_r \sm Y_r^\flat \to Y_0$, which is \Cref{vanish} below.


\section{Fibers over the classical Deligne--Lusztig variety}\label{sec:fibers_over_Y0}

Here we complete the proof of Theorem \ref{thm:concentrates_Howe}. Let the notation be as in \S\ref{sec:DeligneLusztig}. In particular, we have a fixed character $\phi \colon T^F \to \Lambda^\times$ of ${\rm depth}(\phi) \leq r$, admitting a Howe factorization, and the corresponding groups $\BK_\phi, \BK_\phi^+, \dots$, as well the varieties $X_r \supseteq X_r^\flat$, $Y_r \supseteq Y_r^\flat$. We will denote the character induced by $\phi$ on the subquotient $\BT_r^F$ of $T^F$ again by $\phi$. Recall the local system $\CE_\phi$ on $Y_r$ attached to $\phi$. Let
\[ \d_r: Y_r \sm Y_r^\flat \hookrightarrow Y_r \to Y_0\]
be the natural projection. Note that $Y_0$ is (essentially) a classical Deligne--Lusztig variety for $\BG_0$.
\begin{proposition} \label{vanish}
 Suppose the element $w\sigma \in W\sigma$ attached to $F$ in \S\ref{notation:4} is convex. Let $\bar g_0 \in Y_0$. Then we have $R\Gamma_c(\d_r\i(\bar g_0), \CE_\phi) = 0$.
\end{proposition}

After necessary preparations, we prove \Cref{vanish} at the end of \S\ref{sec:fibers_over_Y0}. \emph{Until the end of \S \ref{sec:fibers_over_Y0}, we assume that the element $w\sigma \in W\sigma$ attached to $F$ is convex, so that results of \S\ref{sec:convex} apply; we fix $\bar g_0 \in Y_0 \subseteq \BG_0 / \BT_0$ and a lift $g_0 \in \BG_r$ of $\bar g_0$ such that
\[ y_0 := g_0\i F(g_0) \in \ov \BU_r \cap F\BU_r.\]
}

\subsection{Parametrization of Moy--Prasad quotients} \label{sec:some_setup}
We set
\begin{align*}
\tPhi_r^0 &= \{f \in \tPhi; 0 \le f(\bx) \le r\}; \\ \tPhi_r^+ &= \{f \in \tPhi; 0 < f(\bx) \le r\}; \\ \tD_r &= \{f \in \tPhi_r^0; \a_f \in \Phi_U \cap F\Phi_{\ov U}\}.
\end{align*}
Moreover, we set $\Phi_{\aff, r}^0 = \{f \in \tPhi_r^0; \a_f \in \Phi\}$ and $\Phi_{\aff, r}^+ = \{f \in \tPhi_r^+; \a_f \in \Phi\}$. For $f, f' \in \tPhi_r^0$ we write $f' < f$ if either $f'(\bx) < f(\bx)$ or $f'(\bx) = f(\bx)$ and $f' - f$ is a sum of affine roots in $\tD_r$. We extend this partial order to a total order on $\tPhi_r^0$, and still denote it by $\le$. For $f \in \tPhi_r^0$, we write
$\tPhi_r^f = \{f' \in \tPhi_r^0 \colon f' \geq f \}$.


Note that $\BT_r \to \BT_0$ admits a unique splitting, which we denote by $t \mapsto [t]$. Let $f \in \tPhi_r^+ \cup \{0\}$. Define
\begin{align*}
u_f\colon \BA_f := \BA^1 \to \BT_r\BG_r^{0+}, \quad &x \mapsto U_{\a_f}([x]\varpi^{n_f}) &\text{if \,\, $f \in \tPhi_\aff$,}\\
u_f \colon \BA_f :=\BT_0 \to \BT_r\BG_r^{0+}, \quad &x \mapsto [x] &\text{if \,\,} f = 0,\\
u_f \colon \BA_f := X_*(T) \otimes \ov \BF_q \to \BT_r\BG_r^{0+},\quad &\l \otimes x \mapsto \l(1 + [x]\varpi^{n_f}) &\text{if $f \in \BZ_{\geq 1}$,}
\end{align*}
where in the last line $\l \in X_*(T), x \in \ov \BF_q$.

Define an abelian group $\BA[r] = \prod_{f \in \tPhi_r^+ \cup\{0\}} \BA_f$. Then we have an isomorphism of varieties
\begin{equation}\label{eq:u}
u \colon \BA[r] \to \BT_r\BG_r^{0+}, \quad (x_f)_f \mapsto \prod_f u_f(x_f),
\end{equation}
where the product is taken with respect to the order $\leq$ restricted to $\tPhi_r^+ \cup \{0\}$. Let $E \subseteq \tPhi_r^+\cup\{0\}$. We define $\BA_E = \prod_{f \in E} \BA_f$ which is viewed as a subgroup of $\BA[r]$ in the natural way. We denote by $p_E \colon \BA[r] \to \BA_E$ the natural projection. Define
\[
\BG_r^E = u(\BA_E) \subseteq \BT_r\BG_r^{0+}.
\]
Moreover, we denote by
\[
\pr_E: \BT_r\BG_r^{0+} \cong \BA[r] \to \BA_E \cong \BG_r^E
\]
the natural projection. If
$E + E, \BZ_{\ge 0} + E \subseteq E \cup \tPhi^{r+}$, then $\BG_r^E$ is a subgroup of $G_r^{0+}$.

Let $g \in \BT_r\BG_r^{0+}$, $x \in \BA[r]$ and $E \subseteq \tPhi_r^+\cup \{0\}$. We set $g_E = \pr_E(g) \in u(\BA_E)$, $x_E = p_E(x) \in \BA_E$ and $\hat x = u(x) \in \BG_r^+$. For $f \in \tPhi_r^+$ we will set $x_f = x_{\{f\}}$ and $x_{\ge f} = x_{\tPhi_r^f}$. We can define $g_f$ and $g_{\ge f} \in \BG_r^+$ in a similar way. By abuse of notation, we will identify $g_f \in u(\BA_f)$ with $u\i(g_f) \in \BA_f$ according to the context.

\subsection{Description of the fiber}
Define
\[
\pi \colon \BG_r \to \BG_r, \quad \pi(g) = g\i y_0 F(g).
\]
Let $Y_r(\bar g_0)$, $X_r(\bar g_0)$ be the preimages in $Y_r$, $X_r$ of $\bar g_0 \in Y_0$ under the natural projections. Then we have isomorphisms: $g \mapsto gg_0^{-1}$ induces an isomorphism
\begin{align}
\nonumber Y_r(\bar g_0) &\stackrel{\sim}{\to} \{g \in \BG_r^{0+} \BT_r \colon \pi(g) \in \BT_r (\ov \BU_r \cap F\BU_r)\}/ \BT_r \\
\label{eq:description_fiber} &\stackrel{\sim}{\longleftarrow} \{g \in \BG_r^{\Phi^+_{\aff,r}} \colon \pi(g) \in \BT_r(\ov\BU_r \cap F\BU_r)\}
\end{align}
where the first map is induced by $h \mapsto g_0^{-1}h$, and the second map is induced by $g \mapsto g\BT_r$.
Under these isomorphisms, $\d_r\i(\bar g_0)$ identifies with the subvariety of those $g$ for which $\pi(g) \in \BT_r((\ov \BU_r \cap F\BU_r) \sm \BK_\phi)$.
Until the end of \S\ref{sec:fibers_over_Y0} we will identify $Y_r(\bar g_0)$, $\d_r\i(\bar g_0)$ with the models given by the last line of \eqref{eq:description_fiber}. This enables us to define a map \[ \pi_{\BT_r} \colon Y_r(\bar g_0) \to \BT_r, \quad \pi_{\BT_r}(g) = \pi(g)_{\BT_r}, \]
and we denote its restriction to $\delta_r^{-1}(\bar g_0)$ again by $\pi_{\BT_r}$.


\begin{lemma}\label{lm:cartesian}
There is a cartesian diagram
\[
\xymatrix{
X_r(\bar g_0) \ar[r] \ar[d] & \BT_r \ar[d]^{L_{\BT_r}^{-1}} \\
Y_r(\bar g_0) \ar[r]^{\pi_{\BT_r}} & \BT_r
}
\]
where the left map is the natural projection, $L_{\BT_r}$ is the Lang map of $\BT_r$, and the upper map sends
$h \in X_r(\bar g_0) \cong \{h \in \BG_r^{0+}\BT_r \colon \pi(h) \in \ov \BU_r \cap F\BU_r \}$ to $h_{\BT_r}$.
\end{lemma}
Note that in the diagram of the lemma both horizontal maps depend on the parametrization of $\BG_r^{0+}\BT_r$ fixed in \S \ref{sec:some_setup}.
\begin{proof} As both vertical maps in the diagram are \'etale $\BT_r^F$-torsors, it suffices to show that the diagram commutes. For this, let $h \in X_r(\bar g_0)$. Its image in $Y_r(\bar g_0)$ identifies (under the isomorphism from the previous paragraph) with $h_{\Phi_{\aff, r}^+} = hh_{\BT_r}^{-1}$. Its image under the lower map is then equal to
\begin{align*}
(hh_{\BT_r}^{-1})^{-1}y_0 F(hh_{\BT_r}^{-1}) &= (h_{\BT_r} \cdot (h^{-1}y_0F(h)) \cdot F(h_{\BT_r})^{-1})_{\BT_r} \\
&= h_{\BT_r} F(h_{\BT_r})^{-1} \\
&= L_{\BT_r}(h_{\BT_r})^{-1}
\end{align*}
where the second equality follows as $h\in X_r(\bar g_0)$.
\end{proof}

By proper base change theorem, Lemma \ref{lm:cartesian} implies that
\[
\CE_{\phi}|_{Y_r(\bar g_0)} \cong \pi_{\BT_r}^\ast \CL_\phi,
\]
where $\CL_\phi$ denotes the multiplicative local system on $\BT_r$ corresponding to $\phi$. Clearly, the same isomorphism holds after restricting to $\d_r\i(\bar g_0)$.

Now we prove that the cohomology of the fiber $\d_r\i(\bar g_0)$ with coefficients in $\CE_\phi$ is independent of $r$ as long as $r \geq r_{d-1}$.


\begin{proposition} \label{red-1}
    Let $r \in J$ satisfying $r \ge r_{d-1} > 0$.  Then we have \[R\Gamma_c(\d_r\i(\bar g_0), \pi_{\BT_r}^* \CL_\phi) \cong R\Gamma_c(\d_{r+}\i(\bar g_0), \pi_{\BT_{r+}}^* \CL_\phi)[2m],\] where $m = |\tD_{r+} \sm \tD_r|$.
\end{proposition}
\begin{proof}
    In the setup of \S \ref{sec:Lie_alg}, let $A = \{\a_f \in \Phi\colon f({\bf x}) = r\}$ and let $\phi \colon H_A \to H_A$ be the endomorphism determined by conjugation with $y_0 = \prod_{\alpha \in -\Delta} y_{0,\alpha}$. Note moreover, that \eqref{eq:description_fiber} gives a section $s \colon Y_r(\bar{g}_0)\to Y_{r+}(\bar g_0)$ to the natural projection. The fiber of $Y_{r+}(\bar g_0) \to Y_r(\bar g_0)$ over $g$ is then given by $V(\phi,x,z(s(g)))$, where $z \colon Y_r(\bar g_0) \to \BG_{r+}^{r+}/\BT_{r+}^{r+} \cong H_A$ is some morphism. Then \Cref{uniformization} gives an isomorphism $Y_{r+}(\bar g_0) \cong Y_r(\bar g_0) \times \BA_{\tD_{r+} \sm \tD_r}$. By assumption on $r$, $\d_{r+}\i(\bar g_0)$ is the preimage in $Y_{r+}(\bar g_0)$ of $\d_r\i(\bar g_0) \subseteq Y_r(\bar g_0)$, so that we get an isomorphism
\[\d_{r+}\i(\bar g_0) \cong \d_r\i(\bar g_0) \times \BA_{\tD_{r+} \sm \tD_r}.\] Moreover, for $(x, y) \in \d_r\i(\bar g_0) \times \BA_{\tD_{r+} \sm \tD_r}$ we have
\[
\pi_{\BT_{r+}}(x, y) - \pi_{\BT_r}(x) \in \BT_{\der,r+}^{r+}.
\]
Since the restriction of $\phi$ to $(\BT_{\der,r+}^{r+})^F$ is trivial, $\pi_{\BT_{r+}}^* \CL_\phi$ is isomorphic to the pullback of $\pi_{\BT_r}^* \CL_\phi$ under the natural projection $\BT_{r+} \to \BT_r$. Therefore, we have \[\pi_{\BT_{r+}}^* \CL_\phi \cong \pi_{\BT_r}^*\CL_\phi \boxtimes \Lambda\] and the statement follows by the K\"{u}nneth formula.
\end{proof}

\subsection{Handling jumps}
Let $(G^i, \phi_i, r_i)_{-1 \le i \le d}$ be the Howe factorization of $\phi$ from \S \ref{sec:Howe}. Set
\[
r = r_{d-1}, \qquad M = G^{d-1}, \qquad V = \overline U \cap F U.
\]
We label the real numbers of $\{f(\bx); f \in \tPhi_r^0 \sm \tPhi_M\} \subseteq J$ in the ascending order: \[0 = s_0 < s_1 < \cdots < s_m = r.\] Note that $s_i + s_{m-i} = r$ for $0 \le i \le m$. Set
\begin{align*}
C^0 &= \{f \in \Phi_\aff; f(\bx) = 0\}\\
C^i &= \{f \in \tPhi \sm \tPhi_M; f(\bx) = s_i\} \qquad \text{ for $0 < i < m$}\\
C^m &= \{f \in \tPhi; f(\bx) = r\}.
\end{align*}
Put $C^{<i} = \bigcup_{j<i} C^j$ and define $C^{> i}$ similarly. Note that for for $j \geq \frac{m}{2}-1$, $\BG_r^{C^{>j}} \subseteq \BG_r^{0+}$ is a subgroup normalized by $\BT_r^{0+}$. For $0 \le i, j \le m$ with $j \geq \frac{m}{2}-1$ we define
\begin{align*}
Y_{y_0, r}^{i, j} &:= \{ g \in \BG_r^{\Phi^+_{\aff,r} \cap C^{\leq j}} \colon \pi(g) \in \BT_r^{0+} \BG_r^{C^{>j}} \BV_r^{C^{\ge i}} \}\\
&\cong \{g \in \BG_r^{0+} \colon \pi(g) \in \BT_r^{0+} \BG_r^{C^{>j}} \BV_r^{C^{\ge i}}\} / \BT_r^{0+} \BG_r^{C^{>j}},
\end{align*}
where
\[
\BV_r^{C^{\ge i}} := \{v \in \BV_r \sm \BK_{\phi,r}; v_{C^{<i} \sm \tPhi_M} = 0\}.
\]
Note that $C^{>m} = C^{<0} = \varnothing$ and hence $Y_{y_0, r}^{0, m} = \d_r\i(\bar g_0)$. Moreover, if $0 < i$ then $Y_{y_0, r}^{i, j} \neq \varnothing$ if and only if $y_0 \in \BM_r$.

\begin{lemma} \label{pr1}
    Let $E \subseteq C^{> m/2}$ and $E' \subseteq \tPhi_r^+$ such that $E + (E' \sm \tPhi_M) \subseteq \tPhi^r$. Let $x \in \BA_E$ and $y \in \BA_{E'}$. Then \[ (\hat x \hat y)_{\BT_r} = -\sum_f \a_f^\vee(1 + \varpi^r x_f y_{r-f}) + y_{\BT_r} + x_{\BT_r} \in \BT_r,\] where $f$ ranges over $E$ such that $f > r-f$, and where we denote the group law in $\BT_r$ by $+$.
\end{lemma}

\begin{proof}
The proof is similar to \cite[Lemma 5.13]{IvanovNie_24}.
Write $\hat x \hat y = \hat z_{g_1} \dots \hat z_{g_n} \in \BG_r^{0+}$ with $E'' := \{g_1<\dots <g_n\} \subseteq (\BZ_{\geq 0}E + \BZ_{\geq 0}E') \cap \tPhi_r^+$. Then each $z_{g_i} \in \BA_{g_i}$ is a sum of $x_{g_i}$ (appears if $g_i \in E$), $y_{g_i}$ (appears if $g_i \in E'$) and possibly some iterated commutator terms arising from $x_f$,$y_{f'}$ with $f \in E,f' \in E'$.

As $E \subseteq C^{>m/2}$, we even have $E'' \sm (E\cup E') \subseteq E + \BZ_{\geq 1}E'$. Let $1\leq i \leq n$ be such that $g_i \in \BZ_{\geq 1}$. Suppose that $g_i = f + \sum_{j=1}^s a_j f_j'$ with some $a_j \in \BZ_{\geq 1}$, $f_j' \in E'$. As $M \subseteq G$ is a Levi subgroup, and $g_i \in \tPhi_M$, there must be some $j_0$ with $f_{j_0}' \in E' \sm \tPhi_M$. Then, by assumption, $f+ f_j' \in \tPhi^r$, which forces $j_0=s=1$, $a_1=1$ and $g_i=f+f_1' = r$.

Thus, if $g_i \in \BZ_{\geq 1}$ and $g_i < r$, then $z_{g_i} = x_{g_i} + y_{g_i} = y_{g_i}$ (note that $E \cap \BZ_{\geq 1} = \varnothing$, and thus $x_{g_i} = 0$). When $g_i=r$, then $(\hat x \hat y)_r = -\sum_f \a_f^\vee(1 + \varpi^r x_f y_{r-f}) + y_{r} + x_{r} \in \BT_r^r$, where the sum ranges of the same index set as in the lemma. As $z_{\BT_r} = z_{g_{i_1}} \dots z_{g_{i_r}}$, where $g_{i_j} = j \in \BZ_{\geq 1}$, this finishes the proof.\qedhere

\end{proof}

\begin{lemma} \label{pr2}
    Let $0 \le i \le m/2$ and $w \in \BA_{C^{m-i}}$. For each $f \in C^i$ we have \[(y_0\i \hat w y_0)_{r-f} = \sum_{f \le f' \in C^i} c_{f, f'} w_{r-f'},\] where $c_{f, f'} \in \ov\BF_q$ are some constants such that $c_{f, f} = 1$.

    As a consequence, $(y_0\i \hat w y_0)_{\BT_r} = 0$ if $0< i < m/2$ and $(y_0\i \hat w y_0)_{\BT_r} = \sum_{f \le f' \in C^i} \a_f^\vee(1 + \varpi^r d_{f, f'} (y_0)_f w_{r-f'})$ if $i=0$. Here $d_{f, f'} \in \ov\BF_q$ are some constants such that $d_{f, f'} = 1$.
\end{lemma}
\begin{proof}

Write $\hat w = \hat w_{f_1}\dots \hat w_{f_s}$ with $C^{m-i} = \{f_1<\dots <f_s\}$. It is clear that $y_0^{-1} \hat w y_0$ only depends on the image of $y_0$ in $\BG_0$. By induction on the number of roots in $\tD \cap C^0$ needed to write $y_0$, we may assume that $y_0 = y_{0,g}$ with some $g \in \tD \cap C^0$. We compute
\begin{equation}\label{eq:commutator_aux}
y_0^{-1} \hat w y_0 = \hat w_{f_1} [\hat w_{f_1}^{-1},y_0^{-1}] \hat w_{f_2} \dots \hat w_{f_s} [\hat w_{f_s}^{-1},y_0^{-1}].
\end{equation}
Moreover, $[\hat w_{f_j}^{-1},y_0^{-1}] = \prod_a (c_{a}\hat w_{f_j})_{f_j + ag}$, where the product is taken over all $a\in \BZ_{\geq 1}$ such that $f_j+ag \in \tPhi$ (and hence in $C^{m-i}$), and $c_a \in \overline \BF_q$ is a constant depending on $a,y_0$. If $m-i > m/2$, then all terms $\hat w_{f_j}$, $[\hat w_{f_{j'}}^{-1},y_0^{-1}]$ in \eqref{eq:commutator_aux} commute with each other (in $\BG_r$) and the result follows. If $m-i = m/2$, then the terms in \eqref{eq:commutator_aux} commute up to $\BG_r^r$, which may be ignored, as $r-f$ (from the statement of the lemma) lies in $C^{m/2}$.
\end{proof}

\begin{proposition} \label{product}
    Let $0 \le i \le m/2$. The map $g \mapsto (g_{C^{< m-i}}, g_{C^{m-i} \cap \tD_r})$ induces an isomorphism \[ Y_{y_0, r}^{i, m-i} \cong Y_{y_0, r}^{i, m-i-1} \times \BA_{C^{m-i} \cap \tD_r}.\] Moreover, for $g = (g', z) \in Y_{y_0, r}^{i, m-i-1} \times \BA_{C^{m-i} \cap \tD_r}$ we have
    \begin{itemize}
        \item if $0 \le i < m/2$, then \[\pi(g)_{\BT_r} = \sum_{f \le f' \in C^i \cap -\tD_r} \a_f^\vee(1 + \varpi^r c_{f, f'} \pi(g')_f z_{r-f'}) + \pi(g')_{\BT_r},\] where $c_{f, f'} \in \CO_\brk$ are some constants with $c_{f, f} = 1$;

        \item if $i = m/2$ and $g' \in \BM_r \cap Y_{y_0, r}^{i, m-i-1}$, then \[\pi(g)_{\BT_r} = \mu(z) + \pi(g')_{\BT_r},\] where $\mu: \BA_{C^{m/2} \cap \tD_r} \to \BT_r$ is a certain morphism.
    \end{itemize}
\end{proposition}
\begin{proof}
    By Proposition \ref{uniformization} we have an isomorphism \[ \psi: Y_{y_0, r}^{i, m-i} \stackrel{\sim}{\to} Y_{y_0, r}^{i, m-i-1} \times \BA_{C^{m-i} \cap \tD_r},\] and moreover, for $g = \psi\i(g', z)$ with $(g', z) \in Y_{y_0, r}^{i, m-i-1} \times \BA_{C^{m-i} \cap \tD_r}$ we have $g = g' \hat w$ for some $w \in \BA_{C^{m-i}}$ such that \[\tag{*} w_f = z_f \text{ for } f \in C^{m-i} \cap \tD_r.\]

    We set $h = \pi(g') \in \BT_r^{0+}\BG^{C^{> m-1-i}}_r\BV_r^{C^{\geq i}}$. By definition we have $y_0 = h_{C^0}$, $h_{C^{<i} \sm \tPhi_M} = 0$. Write $h = h_{C^0} h_+ = y_0 h_+$ with $h_+ = h_{\tPhi_r^{0+}}$. Then $h_{\BT_r} = (h_+)_{\BT_r}$.
    We have \[\pi_{\BT_r}(g) = (\hat w\i h F(\hat w))_{\BT_r} = (y_0 y_0\i \hat w y_0 h_+ F(\hat w))_{\BT_r} = (y_0\i \hat w y_0 h_+ F(\hat w))_{\BT_r}.\]

    Assume $i = 0$. By Lemma \ref{pr1} we have \begin{align*} \pi_{\BT_r}(g) = ((y_0\i \hat w y_0) h_+)_{\BT_r} = (y_0\i \hat w y_0)_{\BT_r} +  (h_+)_{\BT_r}. \end{align*} Hence the statement follows from Lemma \ref{pr2} and (*).

    Assume $0 < i < m/2$. Then $y_0\i \hat w y_0  \in \BG_r^{C^{\ge m-i}}$. Moreover, as $0 < i < m-i$, we have $h_+ F(\hat w) \in \BM_r^+ \BG_r^{C^{\ge i}}$ and $(h_+ F(\hat w))_f = h_f = (h_+)_f$ for $f \in C^i$. Applying Lemma \ref{pr1}, Lemma \ref{pr2} and (*) we deduce that \begin{align*}
        \pi_{\BT_r}(g)
        &= ((y_0\i \hat w y_0) (h_+ F(\hat w))_{\BT_r} \\
        &= \sum_{f \in C^i} \a_f^\vee(1 + \varpi^r (y_0\i \hat w y_0)_{r - f} (h_+ F(\hat w))_f) + (y_0\i \hat w y_0)_{\BT_r} + (h_+ F(\hat w))_{\BT_r} \\
        &= \sum_{f \le f' \in C^i \cap -\tD_r} \a_f^\vee(1 + \varpi^r c_{f, f'} h_f w_{r-f'}) + (h_+)_{\BT_r} \\
        &=\sum_{f \le f' \in C^i \cap -\tD_r} \a_f^\vee (1 + \varpi^r c_{f, f'} \pi(g')_f z_{r-f'} ) + \pi_{T, r}(g'),
    \end{align*} where the third equality follows from that $h_f = 0$ for $f \in C^i \sm -\tD_r$.

    Finally, assume that $i = m/2 \in \BZ$ and we may choose $g' \in \BM_r$. Then $h \in \BM_r$ and $h_{C^{\le m/2} \sm \tPhi_M} = 0$. By Proposition \ref{uniformization}, $w \in \BA_{C^{m/2}}$ only depends on $z \in \BA_{C^{m/2} \cap \tD_r}$ (and the fixed element $y_0$). We define $\mu(z) = (y_0\i \hat w y_0 F(\hat w))_{\BT_r}$. Noticing that $y_0\i \hat w y_0 F(\hat w) \in \BG_r^{C^{\ge m/2}}$, $h_+ \in \BM_r^+$ and $[h_+\i, F(\hat w)\i] \in \BG_r^{C^{>m/2} \sm \tPhi_M}$, we deduce by Lemma \ref{pr1} that \begin{align*}
        \pi_{\BT_r}(g)
        &= (y_0\i \hat w y_0 F(\hat w) h_+ [h_+\i, F(\hat w)\i]))_{\BT_r} \\
        &= (y_0\i \hat w y_0 F(\hat w))_{\BT_r} + (h_+ [h_+\i, F(\hat w)\i])_{\BT_r} \\
        &= \mu(z) + (h_+)_{\BT_r} + ([h_+\i, F(\hat w)\i])_{\BT_r} \\
        &= \mu(z) + (h_+)_{\BT_r} \\
        &= \mu(z) + \pi_{\BT_r}(g').
    \end{align*} The proof is finished.
\end{proof}

We have a decomposition \[\BV_r \sm \BK_{\phi,r} = \BV_r' \sqcup \BV_r'',\] where $\BV_r'' = \{g \in \BV_r \sm \BK_{\phi,r}; g_{C^{< m/2} \sm \tPhi_M} = 0\}$ and $\BV_r' = (\BV_r \sm \BK_{\phi,r}) \sm \BV_r''$. This induces a natural decomposition $\d_r\i(\bar g_0) = \d_r\i(\bar g_0)' \sqcup \d_r\i(\bar g_0)''$.

\begin{proposition} \label{criterion}
    Let $\pi: X \times \BG_a \to \BT_r$ be a morphism. Suppose that for each $x \in X$ the pull-back of $\CL_\phi$ via the map $z \mapsto \pi(x, z)$ is isomorphic to a nontrivial multiplicative local system on $\BG_a$. Then $R\Gamma_c(X \times \BG_a, \pi^*\CL_\phi) = 0$.
\end{proposition}
\begin{proof}
Let $x \colon \Spec \ov\BF_q \to X$ be a point, and let $x' \colon \BG_a \to X \times \BG_a$ be the base changed map. Denote by $f \colon X \times \BG_a \to X$ the natural projection and by $f_x$ its pullback along $x$. Proper base change implies that $x^\ast f_!\pi^\ast \CL_\phi \cong f_{x!}x^{\prime\ast}\pi^\ast \CL_\phi$, which is zero by \cite[Lemma 9.4]{Boyarchenko_10}. As this holds for any geometric point $x \in X$, we deduce $f_!\pi^\ast \CL_\phi = 0$. Thus $R\Gamma_c(X\times \BG_a,\pi^\ast\CL_\phi) \cong R\Gamma_c(X,f_!\pi^\ast\CL_\phi) = 0$.
\end{proof}

\begin{proposition} \label{vanish-prime}
    We have $R\Gamma_c(\d_r\i(\bar g_0)', \pi_{\BT_r}^* \CL_\phi) = 0$.
\end{proposition}
\begin{proof}
    Let $m/2 \le j \le m$, $0 \le i < m/2$ and $f \in (C^{\le i} \cap -\tD_r) \sm \tPhi_M$. We define \[Y_{y_0, r}^{f, j} = \{g \in Y_{y_0, r}^{0, j}; \pi(g)_f \neq 0, \pi(g)_{C^{<f} \sm \tPhi_M} = 0\}.\] Then $\d_r\i(\bar g_0)'$ is a disjoint union of locally closed subsets $Y_{y_0, r}^{f, m}$. It suffices to show $R\Gamma_c(Y_{y_0, r}^{f, m}, \pi_{\BT_r}^* \CL_\phi) = 0$.

    By Proposition \ref{product}, we have \[Y_{y_0, r}^{f, m} \cong Y_{y_0, r}^{f, m-1} \times \BA_{C^m \cap \tD_r} \cong \cdots \cong Y_{y_0, r}^{f, m-i-1} \times \BA_{C^{\ge m-i} \cap \tD_r}. \] Moreover, for $g = (g', z) \in Y_{y_0, r}^{f, m-i-1} \times \BA_{C^{\ge m-i} \cap \tD_r}$ we have $\pi(g')_{C^{< f} \sm \tPhi_M} = 0$ and hence \[\pi(g)_{\BT_r} \equiv \a_f^\vee(1 + \varpi^r \pi(g')_f z_{r-f}) + \pi(g')_{\BT_r} \mod (\BT \cap \BM_\der)_r^r.\]
    Since the restriction of the character $\phi$ to $((\BT \cap \BM_\der)_r^r)^F$ is trivial, it follows that the pull-back of $\CL_\phi$ over $\BA_{C^{\ge m-i} \cap \tD_r}$ under the morphism $z \mapsto \pi_{\BT_r}(g', z)$ is isomorphic to the pull-back of $\CL_\phi$ under the morphism $z \mapsto \a_f^\vee (1 + \varpi^r \pi(g')_f z_{r-f})$, which is a nontrivial multiplicative local system. Thus the statement follows from Proposition \ref{criterion}.
\end{proof}

\subsection{Proof of Proposition \ref{vanish}} 

We argue by induction on $d$ and the semisimple rank of $G$. If $d = 0$ or $G$ is a torus, then $V_r = \varnothing$ and the statement is trivial.

Suppose that $d \ge 1$ and hence $M = G^{d-1}$ is a proper Levi subgroup. In view of Proposition \ref{vanish-prime}, it suffices to show $R\Gamma_c(\d_r\i(\bar g_0)'', \pi_{\BT_r}^*\CL_\phi) = 0$. For $m/2-1 \le j \le m$ we define
\[
Y_{y_0, r}^{\prime\prime, j} = \{g \in Y_{y_0,r}^{0,j} \colon \pi(g)_{C^{<m/2} \sm \tPhi_M} = 0 \}
\]
Then $\d_r\i(\bar g_0)'' = Y_{y_0, r}^{\prime\prime, m}$. By Proposition \ref{product}, we have \[Y_{y_0, r}^{\prime\prime, m} \cong Y_{y_0, r}^{\prime\prime, m/2-1} \times \BA_{C^{\ge m/2} \cap \tD}.\] Moreover, for $(g', z) \in Y_{y_0, r}^{\prime\prime, m/2-1} \times \BA_{C^{\ge m/2} \cap \tD}$ we have $\pi_{\BT_r}(g) = \mu(z) + \pi_{\BT_r}(g')$. As $\CL_\phi$ is multiplicative, by K\"{u}nneth formula it suffices to show $R\Gamma_c(Y_{y_0, r}^{\prime\prime, m/2-1}, \pi_{\BT_r}^*\CL_\phi) = 0$.

Indeed, using the natural embedding $\BM_{r-}^{0+}/\BT_{r-}^{0+} \hookrightarrow \BG_r^{0+}/\BT_r^{0+} \BG_r^{C^{\ge m/2}}$ we have \[Y_{y_0, r}^{\prime\prime, m/2-1} = \sqcup_{g \in (\BG_r^{0+})^{\ad(y_0) \circ F} / (\BM_r^{0+} \BG_r^{C^{\ge m/2}})^{\ad(y_0) \circ F}} g Y_{y_0, r-}^M,\] where $Y_{y_0, r-}^M = \{g \in \BM_{r-}^{0+} / \BT_{r-}^{0+}; \pi(g) \in \BT_{r-}^{0+} (\BM_{r-} \cap  \BV_{r-})\}$. Now the statement follows by induction hypothesis that $R\Gamma_c(Y_{y_0, r-}^M, \pi_{T_{r-}}^*\CL_\phi) = 0$. This proves \Cref{vanish} and hence \Cref{thm:concentrates_Howe}.

\section{Explicit description of $R\Gamma_c(X_r^\flat,\Lambda)[\phi]$}\label{sec:comp_with_CS_variety}

We continue to work with notation from \S\ref{sec:DeligneLusztig}. We thus have a  character $\phi \colon T^F \to \Lambda^\times$ of ${\rm depth}(\phi) \leq r$, which we assume to admit a Howe factorization with corresponding subgroups $\BK_\phi, \BK_\phi^+$, etc. We write $\phi$ for the character of $\BT^F$ induced by $\phi$.

\Cref{thm:concentrates_Howe} shows that $R\Gamma_c(X_r,\Lambda)[\phi] = R\Gamma_c(X_r^\flat,\Lambda)[\phi]$. Next, we relate the cohomology of $X_r^\flat$ with the cohomology of a different variety.

Define the subgroup
\[
\BI_{\phi,U, r} = (\BK_{\phi, r} \cap \BU_r)(\BE_{\phi, r} \cap \BT_r)(\BK_{\phi, r}^+ \cap \ov\BU_r).
\]
of $\BK_\phi$ and the subvariety
\begin{align*}
Z_{\phi,U,r} &= \{g \in \BG_r \colon g^{-1}F(g) \in F\BI_{\phi,U, r} \} \\
\end{align*}
acted on by $\BG^F \times \BT^F$ by left and right multiplication. The variety $Z_{\phi,U,r}$ was first considered in a special case by Chen and Stasinski in \cite{ChenS_17}, and later (in general) by the second author in \cite{Nie_24}.
The following result gives a degreewise comparison of the cohomologies of $X_r^\flat$ and $Z_{\phi,U,r}$. This improves over \cite[Theorem 4.1]{Nie_24}, which only compares the ($\BG^F$-equivariant) Euler characteristics.
\begin{proposition}\label{prop:comp_XandZ}
We have a $\BG^F$-equivariant isomorphism
\[R\Gamma_c(X_r^\flat,\Lambda)[\phi] \cong R\Gamma_c(Z_{\phi,U,r},\Lambda)[\phi][2m],
\]
where $m = \dim (\ov \BU_r \cap \BK_{\phi, r}^+) (F\BU_r \cap \BU_r \cap \BK_{\phi, r}) + \dim(\BT_r \cap \BE_{\phi, r})$.
\end{proposition}
This follows directly from \Cref{lm:relationXflat_Xnatural} and \Cref{prop:affine_bundle} below.

\subsection{Proof of \Cref{prop:comp_XandZ}}
Consider
\begin{align*}
X_r^{\flat,\BK} &= X_r^{\flat} \cap \BK_{\phi, r}\\
Z_{\phi,U,r}^\BK &= Z_{\phi,U,r} \cap \BK_{\phi, r},
\end{align*}
both admitting $\BK_{\phi, r}^F \times \BT_r^F$-actions by left/right multiplication. It is immediate that $Z_{\phi,U,r} = \coprod_{\gamma\in \BG_r^F/\BK_{\phi,r}^F} \gamma Z_{\phi,U,r}^\BK$, so that
\begin{equation}\label{eq:Z_is_induced}
R\Gamma_c(Z_{\phi,U,r},\Lambda)[\phi] = \ind_{\BK_\phi^F}^{\BG^F} R\Gamma_c(Z_{\phi,U,r}^\BK, \Lambda)[\phi],
\end{equation}
and the same formulas hold for $X_r^\flat$. To prove \Cref{prop:comp_XandZ} it thus suffices to show $R\Gamma_c(X_r^{\flat,\BK},\Lambda)[\phi] \cong R\Gamma_c(Z_{\phi,U,r}^{\BK},\Lambda)[\phi][2m]$.

Let $\BT_{\phi, r} = \BE_{\phi, r} \cap \BT_r$. Define
\[ X_r^{\natural,\BK} = \{g \in \BK_{\phi, r}; g\i F(g) \in \BT_{\phi, r} (F\BU_r \cap \ov \BU_r \cap \BK_{\phi, r})\}.
\]

\begin{lemma}\label{lm:relationXflat_Xnatural}
We have a natural $\BK_\phi^F$-equivariant isomorphism $R\Gamma_c(X_r^{\flat,\BK}, \Lambda)[\phi] \cong R\Gamma_c(X_r^{\natural,\BK}, \Lambda)[\phi][2\dim\BT_{\phi, r}] $.
\end{lemma}
\begin{proof}
    Since $\BT_{\phi, r}$ is an affine space, the quotient map $X_r^{\natural,\BK} \to X_r^{\natural,\BK} / \BT_{\phi, r}$ induces an isomorphism of $\BK_{\phi, r}^F$-modules
    \[
    R\Gamma_c(X_r^{\natural,\BK}, \Lambda)[\phi] \cong R\Gamma_c(X_r^{\natural,\BK} / \BT_{\phi}, \Lambda)[\phi][2\dim\BT_\phi].
    \]
    On the other hand, there is a natural isomorphism $X_r^{\flat,\BK} / \BT_{\phi, r}^F \to X_r^{\natural,\BK} / \BT_{\phi, r}$. Thus
    \begin{align*}
    R\Gamma_c(X_r^{\natural,\BK} / \BT_{\phi, r}, \Lambda)[\phi] &\cong R\Gamma_c(X_r^{\flat,\BK} / \BT_{\phi,r}^F, \Lambda)[\phi] \\
    &\cong (R\Gamma_c(X_r^{\flat,\BK}, \Lambda) \otimes_{\Lambda\BT_{\phi, r}^F} \Lambda) \otimes_{\Lambda\BT_r^F}^L \phi \\
    \end{align*}
where the second isomorphism follows as $\Lambda\BT_{\phi, r}^F$ is semisimple (or even without using semisimplicity, from \cite[Theorem 2.10]{Dudas_modular}). As $\BT_r^F$ acts freely on $X_r^{\flat,\BK}$, \cite[Lemma 3.2]{BonnafeR_03} shows that $R\Gamma_c(X_r^{\flat,\BK},\Lambda)$ is represented by a complex $P^\bullet \in D(\Lambda(\BG_r^F \times \BT_r^F){-\rm mod})$ such that $P^i$ is $\Lambda\BT_r^F$-projective for each $i\in\BZ$. Then its $\BT_{\phi, r}^F$-coinvariants $P^i \otimes_{\Lambda\BT_{\phi, r}} \Lambda$ is projective as a $\Lambda (\BT_r^F/\BT_{\phi, r}^F)$-module (this follows directly from the definition of projective modules). Note also that $\phi$ is trivial over $\BT_{\phi, r}^F$, hence inflated from some character $\bar\phi$ of $\BT_r^F/\BT_{\phi, r}^F$. From these observations we deduce
\begin{align*}
(R\Gamma_c(X_r^{\flat,\BK}, \Lambda) \otimes_{\Lambda\BT_{\phi, r}^F} \Lambda) \otimes_{\Lambda\BT_r^F}^L \phi &\cong (P^\bullet \otimes_{\Lambda\BT_{\phi, r}^F} \Lambda) \otimes_{\Lambda\BT_r^F}^L \phi \\
&\cong (P^\bullet \otimes_{\Lambda\BT_{\phi, r}^F} \Lambda) \otimes_{\Lambda\BT_r^F}^L \inf\bar \phi \\
&\cong (P^\bullet \otimes_{\Lambda\BT_{\phi, r}^F} \Lambda) \otimes_{\Lambda(\BT_r^F/\BT_{\phi, r}^F)}^L \bar \phi \\
&\cong (P^\bullet \otimes_{\Lambda\BT_{\phi, r}^F} \Lambda) \otimes_{\Lambda(\BT_r^F/\BT_{\phi, r}^F)} \bar \phi \\
&\cong P^\bullet \otimes_{\Lambda\BT_r^F} \phi \\
&\cong R\Gamma_c(X_r^{\flat,\BK}, \Lambda) \otimes_{\Lambda\BT_r^F}^L \phi
\end{align*}
Here the first, the second and the last isomorphisms are by definitions, the fourth follows from $P^i \otimes_{\Lambda\BT_{\phi, r}^F} \Lambda$ being projective as a $\Lambda \BT_r^F/\BT_{\phi, r}^F$-module, and the fifth follows from the usual properties of tensor product of modules. Finally, for the third isomorphism observe that inflation along $\Lambda\BT_r^F \to \Lambda(\BT_r^F/\BT_{\phi, r}^F)$ is $t$-exact and commutes with (derived) tensor products.
This finishes the proof.
\end{proof}

\begin{proposition}\label{prop:affine_bundle}
    The map $(z, a) \mapsto za$ gives an isomorphism \[\varphi_r: X_r^{\natural,\BK} \times (\ov \BU_r \cap \BK_{\phi, r}^+) (F\BU_r \cap \BU_r \cap \BK_{\phi, r}) \overset \sim \longrightarrow Z_{\phi, U, r}^{\BK}.\]
    As a consequence, we have an isomorphism $R\Gamma_c(Z_{\phi, U, r}^{\BK}, \Lambda)[\phi][2m'] \cong R\Gamma_c(X_r^{\natural,\BK}, \Lambda)[\phi]$ as $\BK_{\phi, r}^F$-modules, where $m' = \dim (\ov \BU_r \cap \BK_{\phi, r}^+) (F\BU_r \cap \BU_r \cap \BK_{\phi, r})$.
\end{proposition}
\begin{proof}
First note that $\varphi_r$ is well-defined. Let $z \in Z_{\phi,U,r}^\BK$. It suffices to show there exists a unique $a \in \CA_r:= (F\BU_r \cap \BU_r \cap \BK_{\phi, r}) (\ov \BU_r \cap \BK_{\phi, r}^+)$ such that $z a \in X_r^{\natural,\BK}$. We argue by induction on $r \in \BR_{\ge 0}$. If $r = 0$, then $\CA_r = F\BU_r \cap \BU_r \cap (\BG^0)_r$, $F \BI_{\phi, U, r} = F \BU_r \cap (\BG^0)_r$ and the statement follows from Proposition \ref{Steinberg}.

Suppose that the statement holds for $r-$. We show it also holds for $r>0$. Indeed, by induction hypothesis, there exists $b \in \CA_{r-}$ such that $z b \in X_{r-}^\natural$. Choose a lift of $b$ in $\CA_r$ and still denote it by $b$. Then \[(zb)\i F(zb) \in \BT_{\phi, r} (\BK_{\phi, r} \cap F\BU_r \cap \ov \BU_r) \CD_r\] where $\CD_r = (F\BU_r \cap \BK_{\phi, r} \cap \BG_r^r) (F\ov \BU_r \cap \BK_{\phi, r}^+ \cap \BG_r^r)$.

We assume that $r = r_{i-1}/2$ for some $1 \le i \le d$. The remaining case follows in a simpler way. Let $\Phi_j = \Phi_{G^j} \subseteq \Phi$ be the root system of $G^j$ for $0 \le j \le d$. Then $\CD_r = \CD_r' \oplus \CD_r''$, where $\CD_r'$ (resp. $\CD_r''$) is spanned by the (images) of affine root subgroups of $F(f)$ such that $f(\bx) = r$ and $\a_f \in \Phi_i^+ \sm \Phi_{i-1}$ (resp. $\a_f \in \Phi_{i-1}$). Let $\CC_r = \CA_r \cap \BG_r^r$. Then $\CC_r = \CC_r' \oplus \CC_r''$, where $\CC_r'$ (resp. $\CC_r''$) is spanned by the (images) of affine root subgroups of $f$ such that $f(\bx) = r$ and $\a_f \in (F(\Phi_i^+) \cap \Phi_i^+) \sm \Phi_{i-1}$ (resp. $\a_f \in \Phi_{i-1} \sm (F(\Phi_{i-1}^+) \cap \Phi_{i-1}^-)$). Applying \Cref{Steinberg} and \Cref{uniformization}, there exists $c \in \CC_r$ such that \[c\i ((zb)\i F(zb)) F(c) \in \BT_{\phi, r} (\BK_{\phi, r} \cap F\BU_r \cap \ov \BU_r),\] that is, $z a \in X_r^\natural$ with $a = bc \in \CA_r$ as desired.

Let $a' \in \CA_r$ be another element such that $z a' \in X_r^{\natural,\BK}$. By induction hypothesis, the images of $a$ and $a'$ in $\CA_{r-}$ are the same. Hence we may assume $a' = a y$ for some $y \in \CC_r$. Then it follows from the uniqueness in \Cref{uniformization} that $y$ is trivial and hence $a = a'$ as desired.

The second statement follows from that $\varphi_r$ is $\BK_{\phi, r}^F \times \BT_r^F$-equivariant and that $(\ov \BU_r \cap \BK_{\phi, r}^+) (F\BU_r \cap \BU_r \cap \BK_{\phi, r})$ is an affine space.
\end{proof}

\subsection{Concentration in one degree}
Let $Z_{\phi, U, r}^{\BH} = Z_{\phi, U, r} \cap \BH_{\phi, r}$ and $Z_{\phi, U, r}^{\BL} = Z_{\phi, U, r} \cap \BL_r$ with $\BL = \BG^0$. We set \[ \bar Z_{\phi, U}^{\BH} =  Z_{\phi, U, r}^\BH / \BE_{\phi, r} \quad \text{ and } \quad Z_{\phi, U, r}^{\BL} = Z_{\phi, U, r}^{\BL} / (\BE_{\phi, r} \cap \BL_r).\]

When $\Lambda = \cool$, the following result is proved in \cite[Theorem 6.2]{Nie_24}. With minor modifications the proof goes through for general $\Lambda$.

\begin{theorem} \label{kappa}
The cohomology of $R\Gamma_c(\bar Z_{\phi, U, r}^{\BH}, \Lambda)[\phi|_{(\BT^{0+})^F}]$ is concentrated in a single degree $n_\phi \geq 0$, which is independent of $\Lambda$. Moreover, the $\Lambda$-module $H_c^{n_\phi}(\bar Z_{\phi, U, r}^{\BH}, \Lambda)[\phi|_{(\BT^{0+})^F}]$ is free.
\end{theorem}

\begin{proof}[Proof of \Cref{kappa}] 
Suppose first that $\Lambda \in \{L,\lambda\}$ is a field. Then the proof of \cite[Theorem 6.2]{Nie_24} applies literally with the only difference that instead of \cite[Proposition 6.8(2)]{Nie_24} we have to use \Cref{prop:BW_modular}. Now suppose that $\Lambda= \CO_L$. Let $\bar\phi \colon \BT_r^F \to \lambda^\times$, $\phi[\ell^{-1}] \colon \BT_r^F \to L^\times$ denote the characters induced by $\phi$. By \Cref{lm:integral_coefficients}(2) and (3) it suffices to show that the single non-vanishing cohomology degrees of the complexes $R\Gamma_c(\bar Z_{\phi, U, r}^{\BH}, L)[\phi[\ell^{-1}]|_{(\BT_r^{0+})^F}]$ and $R\Gamma_c(\bar Z_{\phi, U, r}^{\BH}, \lambda)[\bar\phi|_{(\BT_r^{0+})^F}]$ coincide and that the $L$- resp. $\lambda$-dimensions of the cohomology groups are equal. We thus have to compare the proofs of \cite[Theorem 6.2]{Nie_24} for $\Lambda = L,\lambda$ and to show that they output vector spaces of equal dimension, sitting in the same degree. We will now use the notation of \cite[\S6]{Nie_24}. For an object $?$ (like $\BH, D, \dots$) from \cite[\S6]{Nie_24}, write $?_{L}$ resp. $?_{\lambda}$ depending on the coefficients we consider. Note that $\ker(\CO_L^\times \to \lambda^\times) = \mu_{\ell^\infty}(L)$ intersects the image of $\phi|_{(\BT_r^+)^F}$ trivially. In particular, $\BH_L = \BH_\lambda$, $\BE_L = \BE_\lambda$ and hence also $\bar\BH_L = \bar\BH_\lambda$. Moreover, $D_L = D_\lambda$, $D_{m,L} = D_{m,L}$, $D_{s,L}^{\pm} = D_{s,\lambda}^{\pm}$. Thus $\bar\BH_L^{\natural} = \bar\BH_\lambda^{\natural}$ and hence $\bar Z_L^{\natural} = \bar Z_\lambda^\natural$ and $\bar Z_L^{\flat} = \bar Z_\lambda^\flat$. Continuing further in the same way, it remains to show that with notation as in \cite[Lemma 6.10]{Nie_24}, $n_{C,L} = n_{C,\lambda}$ and $\dim_L H_c^{n_{C,L}}(Y \cap A^C, \varphi^\ast\CL_{\phi,L})  = \dim_\lambda H_c^{n_{C,\lambda}}(Y \cap A^C, \varphi^\ast\CL_{\phi,\lambda})$ of the same dimension sitting in the same cohomological degree. Using \Cref{prop:BW_modular}, this follows from \cite[Proof of Lemma 6.10]{Nie_24}, noticing that nothing in the proof depends on the coefficients.
\end{proof}

By \cite[\S7.1]{Nie_24}, each cohomology group $H_c^i(Z_{\phi, U, r}^{\BH}, \Lambda)[\phi|_{(\BT_r^{0+})^F}]$ is equipped a natural $\BK_\phi^F$-module structure, and we define 
\[
\k_\phi = (-1)^{n_\phi} H_c^{n_\phi}(\bar Z_{\phi, U, r}^{\BH}, \Lambda)[\phi|_{(\BT_r^{0+})^F}] 
\]
as a $\Lambda\BK_{\phi, r}^F$-module (which is $\Lambda$-free), extending its natural $\BH_{\phi, r}^F$-module structure. We have the following extension of \cite[Proposition 7.4 and Theorem 7.5]{Nie_24} to the modular case.

\begin{theorem} \label{degreewise}
     Let $T \subseteq B$ be as in Proposition \ref{standard} with $M = L = G^0$. Then we have an isomorphism in $D^b(\Lambda\BK_{\phi, r}^F\!-\!{\rm mod})$, \[R\Gamma_c(\bar Z_{\phi, U, r}^{\BK}, \Lambda)[\phi] \cong (-1)^{n_\phi}\k_\phi[0] \otimes_{\Lambda} R\Gamma_c(\bar Z_{\phi, U, r}^{\BL}, \Lambda)[\phi_{-1}][-n_\phi],\] where $n_\phi$ is as in Theorem \ref{kappa}, and $(-1)^{n_\phi}\k_\phi[0]$ is concentrated in degree $0$.
\end{theorem}
\begin{proof}
    By the assumption on $B$ we have that $\bar \BH_{\phi, r} \cap F \bar\BI_{\phi, U, r}$ is normalized by $\bar \BL_r$. By \cite[Proposition 7.4]{Nie_24} the map $(h, l) \mapsto h l$ induces a $(\BT_r^{0+})^F/(\BE_{\phi, r} \cap \BT_r)^F$-torsor \[f: \bar Z_{\phi, U, r}^{\BH} \times \bar Z_{\phi, U, r}^{\BL} \to \bar Z_{\phi, U, r}^{\BK}.\] Let $\phi^\flat$ be the pull-back of the natural multiplication map $(\BT_r^{0+})^F \times \BT_r^F \to \BT_r^F$. Let $\phi^\flat = \phi \circ m$, where $m \colon (\BT_r^{0+})^F \times \BT_r^F \to \BT_r^F$ sends $(t_1,t_2) \mapsto t_1t_2$. Combining Theorem \ref{kappa} and the arguments in the proof of \emph{loc. cit.} we deduce that 
    \begin{align*}
        &\quad\ R\Gamma_c(\bar Z_{\phi, U, r}^{\BK}, \Lambda)[\phi] \\
        &\cong R\Gamma_c(\bar Z_{\phi, U, r}^{\BH} \times \bar Z_{\phi, U, r}^{\BL}, \Lambda)[\phi^\flat] \\
        &\cong R\Gamma_c(\bar Z_{\phi, U, r}^{\BH}, \Lambda)[\phi|_{(\BT_r^{0+})^F}] \otimes_\Lambda R\Gamma_c(\bar Z_{\phi, U, r}^{\BL}, \Lambda)[\phi] \\
        &\cong (-1)^{n_\phi} \k_\phi[0] \otimes_\Lambda R\Gamma_c(\bar Z_{\phi, U, r}^{\BL}, \Lambda)[\phi_{-1}][-n_\phi].  
    \end{align*} 
    The proof is finished.
\end{proof}

\begin{corollary}\label{cor:concentration_in_general}
    Let $T \subseteq B$ be as in Proposition \ref{standard} with $M = L = G^0$. Suppose that $\phi_{-1}$ is non-singular for the special fiber of $(\CG^0)_{\bf x}$ in the sense of \cite[Definition 5.15]{DeligneL_76}. Then the cohomology of $R\Gamma_c(X_r,\Lambda)[\phi]$ is non-zero in exactly one degree, $N_\phi \geq 0$. Moreover, its $N_\phi$th cohomology group, $H_c^{N_\phi}(X_r, \Lambda)[\phi]$, is $\Lambda$-free.
\end{corollary}
\begin{proof}
    By \Cref{thm:concentrates_Howe}, Proposition \ref{prop:comp_XandZ} and that the quotient map $Z_{\phi, U, r}^\BK \to \bar Z_{\phi, U, r}^\BK $ is a $\BK_{\phi, r}^F \times \BT_r^F$-equivariant affine space bundle, it suffices to consider $R\Gamma_c(\bar Z_{\phi, U, r}^{\BK}, \L)[\phi]$. By Theorem \ref{degreewise}, it suffices to show that the cohomology of $R\G_c(\bar Z_{\phi, U, r}^\BL, \L)[\phi_{-1}]$ is concentrated in a single degree and $\Lambda$-free.

    Note that $\bar Z_{\phi, U, r}^\BL / \bar \BT_r^F$ is a classical Deligne-Lusztig variety for the special fiber of $(\CG^0)_{\bf x}$. In view of our assumptions on $\phi_{-1}$, it follows from \cite[Lemma 9.14]{DeligneL_76} and \cite[Lemme 3.5]{Broue90} that $H_c^i(\bar Z_{\phi, U, r}^{\BL}, \L)[\phi_{-1}] \neq 0$ if and only if $i = \dim \bar Z_{\phi, U, r}^{\BL}$. So the statement follows.
\end{proof}

\subsection{Irreducibility}
We assume that $\Lambda = \lambda$ is an algebraically closed field of characteristic $\ell>0$. Let $\phi \colon T^F \to \lambda^\times$ be a smooth character as in the beginning of \S\ref{sec:comp_with_CS_variety}. To apply our results to the determination of the Fargues--Scholze parameters via \cite[Corollary 10.4.2]{Feng_24}, we need to study when the representation 
\begin{equation}\label{eq:def_pi_TUphi} \pi_{T,U,\phi} := {\rm c}\text{-}{\rm ind}_{Z^F \BG^F}^{G^F} R\Gamma_c(X_r,\lambda)[\phi]
\end{equation}
is an irreducible supercuspidal representation of $G^F$ up to sign. 

\begin{remark} By \Cref{cor:concentration_in_general}, when $\phi_{-1}$ is non-singular for the special fiber of $(\CG^0)_{\bf x}$, $R\Gamma_c(X_r,\lambda)[\phi]$ is concentrated in one degree and $\pi_{T,U,\phi}$ is (up to a sign and a shift) a usual $G^F$-representation. Moreover, note that the pair $(T, \phi)$ is tame elliptic regular in the sense of \cite[Definition 3.7.5]{Kaletha_19} (with coefficients $\BC$ there replaced by $\lambda$) if and only if $\phi_{-1}$ is regular in the sense of \cite[Definition 3.4.16]{Kaletha_19}. 
\end{remark}

\begin{proposition}{\cite[Proposition 7.3]{Nie_24}}\label{lm:kappa_phi_irred} 
Regarded as a $\BH_{\phi, r}^F$-module (up to a sign) $\kappa_\phi$ is irreducible. In particular, it is irreducible as a $\BK_{\phi, r}^F$-module. The subgroup $(\BK_{\phi, r}^+)^F$ acts on $\kappa_\phi$ via a character $\phi^\natural$, obtained from $\phi$ as in \cite[\S3.2]{Nie_24}. Moreover, if $\BH_{\phi, r}^F/\ker(\phi^\natural)$ is non-trivial, then $\kappa_\phi|_{\BH_{\phi, r}^F}$ the inflation of the unique irreducible mod-$\ell$ Heisenberg representation with non-trivial central character $\phi^\natural|_{\BH_{\phi, r}^F/\ker(\phi^\natural)}$.
\end{proposition}
\begin{proof}
To emphasize the coefficients $\Lambda$, we write $\kappa^\Lambda_\psi$ instead of $\kappa_\psi$ for a character $\psi \colon T^F \to \Lambda^\times$. Let $\hat\phi \colon T^F \to \CO_L^\times$ denote the Teichm\"uller lift of $\phi$ and denote by $\hat\phi_{\ell^{-1}}$ its composition with $\CO_L^\times \to L^\times$. We have $\BH_\phi = \BH_{\hat\phi} = \BH_{\hat\phi_{\ell^{-1}}}$ and similarly for the other attached subgroups. 
By \Cref{kappa} we have 
\[
\kappa_{\hat\phi}^{\CO_L} \otimes_{\CO_L} \lambda \simeq \kappa_{\phi}^{\lambda} \quad \text{ and } \quad \kappa_{\hat\phi}^{\CO_L} \otimes_{\CO_L} L \simeq \kappa_{\hat\phi_{\ell^{-1}}}^L.
\]
Thus $\kappa_{\phi}^\lambda$ is the reduction mod $\ell$ of $\kappa_{\hat\phi_{\ell^{-1}}}^L$. By \cite[Proposition 7.3]{Nie_24}, $\kappa_{\hat\phi_{\ell^{-1}}}^L$ is the Heisenberg representation of (the appropriate quotient of) $\BH_{\hat\phi_{\ell^{-1}}, r}^F$ over $L$. Thus, by \cite[Lemma 2.5]{Fintzen_Michigan}, $\kappa_{\phi}^\lambda$ is the Heisenberg representation of (the same quotient of) $\BH_{\phi, r}^F$ over $\lambda$, and all the properties claimed in the proposition follow from the properties of the Heisenberg representation.\qedhere
\end{proof}

Let $L=G^0$ and let $\widetilde \BL$ be the stabilizer in $L = L(\brk)$ of the image of $\bx$ in the reduced Bruhat-Tits building of $G^0$ over $\brk$. Let $\widetilde\k(\phi)$ denote the restriction of the mod-$\ell$ Weil-Heisenberg representation of  $\widetilde\BK_\phi^F$ associated to $(G_i,r_i, \phi_i)_{i=0}^d$, see \cite{Fintzen_Michigan}. We set $\k(\phi) = \widetilde\k(\phi)|_{Z^F \BK_\phi^F}$.

Let $\widetilde\BK_{\phi} = \widetilde\BL \BK_\phi = \widetilde\BL \BH_\phi$. Attached to $\phi_{-1}$, we have the $\BL_0^F$-representation
\[
\rho_{0,\phi_{-1}} := R\Gamma_c(\bar Z_{\phi, U, r}^{\BL}, \lambda)[\phi_{-1}]
\]
as in \Cref{degreewise}. We regard it as a $Z^F \BL^F$- resp. $Z^F\BK_\phi^F$-representation on which $Z^F$ acts via the character $\phi$. Note that $\bar Z_{\phi, U, r}^{\BL}$ is a (classical) Deligne--Lusztig variety for $\BL_0$. Hence, if $\phi_{-1}$ is regular for $\BL_0$ in the sense of \cite[Definition 3.4.16]{Kaletha_19}, $\rho_{0,\phi_{-1}}$ is (up to sign) an irreducible $\lambda\BL_0^F$-module by \cite[Corollaire 3.6]{Broue90}. 

\begin{lemma}\label{lm:twist_character}
    There is a character $\chi$ of $Z^F \BK_\phi^F$, which is trivial over $Z^F \BH_\phi^F$, such that $\pm\k_\phi \cong \k(\phi) \otimes \chi$ as $Z^F \BK_\phi^F$-modules.
\end{lemma}
\begin{proof}
    Note that $Z^F \BH_\phi^F$ is a normal subgroup of $Z^F\BK_\phi^F$ and the restrictions of $\pm\k_\phi$ and $\k(\phi)$ to $Z^F \BH_\phi^F$ are that same (inflated) Heisenberg representation. By Frobenius reciprocity, we have \[{\rm Hom}_{Z^F\BK_\phi^F}(\pm\k_\phi, \ind_{Z^F \BH_\phi^F}^{Z^F \BK_\phi^F} (\k(\phi)|_{Z^F \BH_\phi^F})) \cong {\rm Hom}_{Z^F \BH_\phi^F}(\pm\k_\phi, \k(\phi)) \neq 0.\] Moreover, by projection formula, we have \[\ind_{Z^F \BH_\phi^F}^{Z^F \BK_\phi^F} (\k(\phi)|_{Z^F \BH_\phi^F}) \cong \k(\phi) \otimes \ind_{Z^F \BH_\phi^F}^{Z^F \BK_\phi^F} 1.\]  As $Z^F \BH_\phi^F$ acts irreducibly on $\k(\phi)$, by a theorem of Burnside, $\End_{\lambda}(\k(\phi))$ is generated by the group algebra $\lambda \BH_\phi^F$. As also $Z^F \BH_\phi^F$ acts trivially on $\ind_{Z^F \BH_\phi^F}^{Z^F \BK_\phi^F} 1$, it follows that each $Z^F \BK_\phi^F$-submodule of $\k(\phi) \otimes \ind_{Z^F \BH_\phi^F}^{Z^F \BK_\phi^F} 1$ is of the form $\k(\phi)|_{Z^F \BK_\phi^F} \otimes \pi$, where $\pi$ is a submodule of $\ind_{Z^F \BH_\phi^F}^{Z^F \BK_\phi^F} 1$. In particular, $\pm\k_\phi \cong \k(\phi) \otimes \chi$ for some submodule $\chi$ of $\ind_{Z^F \BH_\phi^F}^{Z^F \BK_\phi^F} 1$. As $\pm\k_\phi, \k(\phi)$ are of the same dimension, it follows that $\chi$ is one-dimensional and hence a character as desired.
\end{proof}




\begin{proposition}\label{prop:irred_modular}
Assume $\phi$ is toral, that is, $L = G^0 = T$. Then $\pm\pi_{T,U,\phi}$ is an irreducible supercuspidal $\lambda G^F$-module.
\end{proposition}

\begin{proof} As $L = T$, we have $\widetilde \BL = Z \BL$, $\widetilde \BK_\phi = Z \BK_\phi$ and $\rho_{0, \phi_{-1}} = \phi_{-1}$. Then there exist isomorphisms (up to degree shifts and signs) of $G^F$-representations
\begin{align}
\nonumber \pm\pi_{T,U,\phi} &\cong {\rm c}\text{-}{\rm ind}_{Z^F\BK_{\phi}^F}^{G^F} R\Gamma_c(Z_{\phi,U,r}^{\BK},\lambda)[\phi] \\
\label{eq:pi_TUphi} &\cong {\rm c}\text{-}{\rm ind}_{Z^F\BK_{\phi}^F}^{G^F}  \k_\phi \otimes \rho_{0,\phi_{-1}} \\
\nonumber &\cong {\rm c}\text{-}{\rm ind}_{Z^F \BK_{\phi}^F}^{G^F}  \k_\phi \otimes (\chi \otimes \phi_{-1}) \\ 
\nonumber &= {\rm c}\text{-}{\rm ind}_{\widetilde \BK_{\phi}^F}^{G^F} \widetilde\k(\phi) \otimes (\chi \otimes \phi_{-1})
\end{align}
where $\k_\phi$ and $\chi$ are as in \Cref{degreewise} and \Cref{lm:twist_character} respectively. Indeed, the first equation follows from \Cref{thm:concentrates_Howe} and \Cref{prop:comp_XandZ}, the second from \Cref{degreewise}, and the third one from \Cref{lm:twist_character}.

Therefore, $((G^i,r_i, \phi_i)_{i=0}^d, \bx, \chi \otimes \phi_{-1})$ is a modular generic cuspidal datum as in \cite[\S2.2]{Fintzen_Michigan}, and $\pm \pi_{T, U, \phi}$ is the associated irreducible supercuspidal representation, see \cite[\S3]{Fintzen_Michigan}. This finishes the proof. \qedhere

\end{proof}

\subsection{Cohomology of $Z_{\phi,U,r}^{\BK}$ with coefficients in $\ov\BQ_\ell$} \label{subsec: product} 
In this section we take $\L = \ov\BQ_\ell$. Consider the following virtual $\BG_r^F$-module \[\CR_{T, U, r}^G(\phi) = \sum_{i \in \BZ} (-1)^i H_c^i(Z_{\phi, U, r}, \L)[\phi].\] The main purpose of this subsection is to show that $\CR_{T, U, r}^G(\phi)$ is independent of the choice of $U$, which extending the results of \cite[\S5.2]{Nie_24}. 

Let $(V, \ov V)$ and $(U, \ov U)$ be two pairs of opposite maximal unipotent subgroups of $G$ normalized by $T$. For a subset $R \subseteq \BK_{\phi, r}$, write $\bar R$ for the image of $R$ under the natural projection $\BK_{\phi} \to \BK_{\phi, r}/\BE_{\phi, r}$. As $\phi$ and $r$ are fixed, we omit them from notation in this subsection, and write $\BK, \BH, \BE$ instead of $\BK_{\phi, r},\BH_{\phi, r}, \BE_{\phi, r}$ and so on.
Write $L = G^0$. First note that \[\bar \BL = \sqcup_{w \in W_{\bar \BL}(\bar \BT)} \bar \BL_V \dw \bar \BT \bar \BL_{U},\] where $\bar \BL_V = \bar \BL \cap \BV$ and $\bar \BL_U = \bar \BL \cap \bar \BU$.

For $\alpha \in \Phi$, define $i(\a)$ to be the integer $0\leq i\leq d$, such that $\a \in \Phi(G^i,T) \sm \Phi(G^{i-1},T)$, and define $r(\a) = r_{i(\a) - 1}$. Put
\[ \bar \BH^\a = (\BG^\a)_r^{r(\a)/2} / (\BG^\a)_r^{r(\a)+}.\]
Then we have \[\bar \BH = \bar \BH_V \bar \BT^{0+} \bar \BH_{\ov V} = \bar \BT^{0+} \oplus_\a \bar \BH^\a,\] where $\bar \BH_V = \bar \BH \cap \bar V$, $\bar \BH_{\ov V} = \bar \BH \cap \bar {\ov V}$. For $\a, \b \in \Phi$ we have $[\bar \BH^\a, \bar \BH^\b] = \{0\}$ if $\a \neq -\b$ and $[\bar \BH^\a, \bar \BH^\b] =(\bar \BT^\a)^{r(\a)} \cong (\BT^\a)_r^{r(\a)} / (\BT^\a)_r^{r(\a)+}$ if $\a = -\b$ and $\bar \BH^\a \neq \{0\}$.

Thus we have \[\bar \BK = \bar \BH \bar \BL = \bigsqcup_{w \in W_{\BL_r}(\BT_r)} \bar \BH \bar \BL_V \dw \bar \BT \bar \BL_U = \bigsqcup_{w \in W_{\BL_r}(\BT_r)} \bar \BK_V \bar \BH_{\ov V, w} \dw \bar \BT \bar \BH_U,\] where $\bar \BK_V = \bar \BL_V \bar \BH_V = \bar \BK \cap \bar \BU$ and $\bar \BH_{\ov V, w} = \bar \BH_{\ov V} \cap {}^\dw \bar{\ov \BU}$.


Let $A, A'$ be two finite subgroup with a character $\chi: A \to \ov\BQ_\ell^\times$. Let $Z$ be an $A \times A'$-variety. We set $H_c^*(Z, \ov \BQ_\ell)[\chi] = \sum_{i\in \BZ} H_c^i(Z, \ov \BQ_\ell)[\chi]$ which is  virtual $A'$-module. Here $H_c^i(Z, \ov \BQ_\ell)[\chi]$ is the $\chi$-isotropic space of the $i$th cohomology $H_c^i(Z, \ov \BQ_\ell)$ with compact support.

\begin{proposition}
    We have \[\<H_c^*(Z_{\phi, V, r}^{\BK}, \ov \BQ_\ell)[\phi], H_c^*(Z_{\phi, U, r}^{\BK}, \ov \BQ_\ell)[\phi]\>_{\BK_{\phi, r}^F} = |\stab_{W_{\BL_r}(\BT_r)^F}(\phi|_{\BT_r^F})|.\] In particular, $\CR_{T, U, r}^G(T)$ is independent of the choice of $U$.
\end{proposition}
\begin{proof}

    For $w \in W_{\bar \BL}(\bar \BT)$ with a lift $\dot w$, we set \[\Sigma_w = \{(x, x', v, \bar v, \t, u) \in F\bar \BK_V \times F\bar \BK_U \times \bar \BK_{V, w} \times \bar \BH_{\ov V, w} \times \bar \BT \times \BK_U; x F(\bar v \dw \t) = v \bar v \dw \t u x'\}.\]
    Write $\Sigma_w = \Sigma_w' \sqcup \Sigma_w''$, where $\Sigma_w''$ is defined by condition that $\bar v = 0$.

    Let $D = \{(t, s) \in \bar\BT \times \bar\BT; t\i F(t) = s\i F(s)\}$. Then $D$ acts on $\Sigma_w''$ by \[(t, s): (x, x', v, \bar v, \t, u) \mapsto (txt\i, sx's\i, svs, s\bar v s\i, \dw\i(t) \t s\i, sus\i).\] It follows that \[(\Sigma_w'')^{D_\red^\circ} \cong (\dw \bar \BT)^F.\] Hence  $H_c^*(\Sigma_w'', \ov\BQ_\ell)[\phi\i \boxtimes \phi] = \ov\BQ_\ell$ if $w = F(w)$ and is trivial otherwise.

    It remains to show $H_c^*(\Sigma_w', \ov \BQ_\ell)[\phi\i \boxtimes \phi] = \ov\BQ_\ell = 0$. Note that \[\bar \BH_{\ov V, w} = \oplus_{\a \in \Phi_{\ov V} \cap {}^w \Phi_{\ov U}} \bar \BH^\a,\] where $\bar \BH^\a = \bar \BH \cap \bar \BG^\a$. For $\bar v \in \bar \BH_{V, w}$ and $\a \in \Phi_{\ov V} \cap {}^w \Phi_{\ov U}$ let $\bar v_\a \in \bar \BH^\a$ such that $\bar v = \sum_\a \bar v_\a$. We fix a total order $\le$ on $\Phi_{\ov V} \cap {}^w \Phi_U$. Let $\bar \BH_{\ov V, w}^\a$ be subset of elements  $\bar v$ such that $\bar v_\a \neq 0$ and $\bar v_\b = 0$ for all $\b < \a$. Then we have \[\bar \BH_{\ov V, w} - \{0\} = \coprod_\a \bar \BH_{\ov V, w}^\a.\] The above decomposition induces a decomposition \[\Sigma_w' = \coprod_\a \Sigma_w^{ \prime, \ge \a}.\] It remains to show $H_c^*(\Sigma_w^{\prime, \a}, \ov \BQ_\ell)_{\phi, \phi\i} = 0$ for all $\a$.

    Let $\a \in \Phi_{\ov V} \cap {}^w \Phi_{\ov U}$ such that $\Sigma_w^{\prime, \a} \neq \varnothing$. Consider the restricted action of $\bar\BT^F \cong \bar\BT^F \times \{1\} \subseteq \bar\BT^F \times \bar\BT^F$ on $\Sigma_w^{\prime, \a}$ given by \[t: (x, x', v, \bar v, \t, u) \mapsto (t x t\i, x', t v t\i, t \bar v t\i, w\i(t)\t, u).\]  It suffices to show
    the $\phi$-isotropic subspace $H_c^*(\Sigma_w^{\prime, \a}, \ov \BQ_\ell)[\phi]$ is trivial.

    For $\bar v \in \bar \BH_{\ov V, w}^{\ge a}$ we fix an isomorphism \[\l_{\bar v}: \bar \BH^{-\a} \overset \sim \longrightarrow (\bar \BT^\a)^{r(\a)}, \quad \z \to [\bar v, \z].\] Let \[\CH = \{t \in \bar \BT^{r(\a)}; t\i F\i(t) \in (\bar \BT^\a)^{r(\a)}\}.\] For $t \in \CH$ we define an isomorphism $f_t: \Sigma_w^{\prime, \a} \to \Sigma_w^{\prime, \a}$ by \[f_t: (x, x', v, \bar v, \t, u) \mapsto (x_t, x' F({}^{(\dw \t)\i} \z), t v t\i, t \bar v t\i, w\i(t) \t, u)\] with $\z = \l_{\bar v}\i(t F\i(t)\i)$ such that \[x_t F(\bar v \dw \t) = t v \bar v \dw \t u x' F({}^{(\dw \t)\i} \z).\] The induced map of $f_t$ on each subspace $H_c^i(\Sigma_w^{\prime, \a}, \ov \BQ_\ell)$ is trivial for $t \in N_F^{F^n}((\bar \BT^\a)^{F^n}) \subseteq \CH^\circ \cap ((\bar \BT^\a)^{r(\a)})^F$. Here $n \in \BZ_{\ge 1}$ such that $F^n(\bar\BT^\a) = \bar\BT^\a$, and $N_F^{F^n}: \bar \BT \to \bar \BT$ is the map given by $t \mapsto t F(t) \cdots F^{n-1}(t)$. On the other hand, we have \[\phi |_{N_F^{F^n}((\bar \BT^\a)^{F^n}} = \phi_{i(\a) - 1} |_{N_F^{F^n}((\bar \BT^\a)^{F^n}},\] which is nontrivial since $\phi_{i(\a)-1}$ is $(G^{i(\a)-1}, G^{i(\a)})$-generic. Here $i(\a)$ is the integer $0 \le i \le d$ such that $\a \in \Phi_{G^i} \sm \Phi_{G^{i-1}}$. Thus $H_c^*(\Sigma_w^{\prime, \a}, \ov \BQ_\ell)[\phi]$ is trivial as desired.
\end{proof}

\section{Pro-unipotent DL-variety for an elliptic torus}\label{sec:prounipotent}

Let $\BG^+ = L^+\CG^{0+}_{\bf x}$ be the pro-unipotent radical of $\BG = L^+\CG_{\bf x}$ (this corresponds to $r=\infty$ in the notation of \S\ref{notation:2}). Consider the perfect scheme
\[
X^+ = \{g \in \BG^+ \colon g^{-1}F(g) \in \ov\BU^+ \cap F\BU^+\}
\]
which is the the inverse limit of its truncations $X_r^+ \subseteq \BG^+_r := \BG^{0+}_r$. Then $X^+$ is acted on by $(\BG^+)^F \times (\BT^+)^F$ by left and right multiplication. In this section we are going to prove \Cref{thm:generalization_first_article}.

\subsection{Preparations}
Fix a total order on $\tPhi^+/\langle F\rangle$ such that $O<O'$, if either $O({\bf x}) < O'({\bf x})$ or ($O({\bf x}) = O'({\bf x})$ and $O \in \BZ_{\geq 1}$, $O' \not\in \BZ_{\geq 1}$). As $T$ is elliptic, any orbit $O \in \tPhi^+/\langle F\rangle$ intersects $\tD^+$, where $\D = \Phi^+ \cap F\Phi^-$. For each orbit $O$, pick some $f \in O \cap \tD^+$ and extend the order to a total order on $\tPhi^+$ in the unique way such that $f < F(f) < \dots<F^{\#O - 1}(f)$. For $f\in \tPhi^+$, denote by $O_f$ its $F$-orbit, and denote by $f+$ (resp. $f-$) any member of the orbit, which is the ascendant (resp. descendant) of $O_f$ with respect to the order on $\tPhi^+/\langle F \rangle$.

We use the setup from \cite[\S5.1-2]{IvanovNie_24}, which slightly differs from that of  \S\ref{sec:some_setup}. In this section for $f\in \tPhi^+$ we put
\[
\tPhi^f = \{f' \in \tPhi^+ \colon O_{f'} \geq O_f \}.
\]
Note that if $f=r\in \BZ_{\geq 1}$, then $\tPhi^f = \{f' \in\tPhi^+ \colon 0<f'({\bf x}) <r\}$, so our notation is not ambiguous. We let $\tPhi_f^+ = \tPhi^+ \sm \tPhi^f$. We let $\BG^f\subseteq \BG^+$ be the subgroup generated by the affine roots subgroups in $\tPhi^f$. It is easy to see that $\BG^f \subseteq \BG^+$ is normal and we put
\[\BG_f^+ = \BG^+/\BG^f.\]
Note that $\tPhi^f,\tPhi^+_f$ are $F$-stable, so that $\BG^f,\BG^+_f$ are defined over $\BF_q$.

Fix some $r\geq 1$. Let $\BA[r] = \prod_{f \in \tPhi_r^+} \BA_f$ (with $\BA_f$ as in \S\ref{sec:some_setup}). As in \eqref{eq:u} we have the isomorphism of varieties
\[
u\colon  \BA[r] \overset \sim \to \BG_r^+, \quad (x_f)_f \mapsto \prod_f u_f(x_f),
\]
where the product is taken with respect to the fixed order on $\tPhi^+$. For a subset $E \subseteq \tPhi^+_r$, set $\BA_E = \prod_{f \in E} \BA_f$, let $p_E \colon \BA[r] \to \BA_E$ denote the natural projection, and let $\pr_E \colon \BG_r^+ \to u(\BA_E)$ denote the map obtained from $p_E$ by transport of structure via $u$. For $f \in \tPhi_r^+$, write $p_f = p_{\{f\}}$ and $\pr_f = \pr_{\{f\}}$. When the context is clear, we sometimes will abuse the notation and identify $\pr_f \colon \BG^+_r \to u(\BA_f)$ with $u^{-1}\circ \pr_f \colon \BG_r^+ \to \BA_f$.

Let $f \in -\widetilde\D^+$. Then there exists a unique sequence
\begin{equation}\label{eq:sequence_of_sign_changes} 0 = a_0 < a_1 < \dots < a_{2b(f)}=\#O_f \end{equation}
of integers, such that $F^{a_{2i}}(f) \in \widetilde\Pi^+$, $F^{a_{2i-1}}(f) \in \widetilde\Delta^+$ for all $0\leq i \leq b(f)$, and $F^k(f) \not\in \widetilde\Delta^+ \cup -\widetilde\D^+$ if $k \not\equiv a_j \mod \#O_f$ for any $j$.



\subsection{A cartesian diagram} Fix some $r \in \BZ_{\geq 1}$. Let $\widetilde\Phi^r \subseteq B \subseteq A \subseteq \widetilde\Phi^+$ be two closed subsets with $A+B \subseteq B$, $\BZ_{\geq 0} + A \subseteq A$, $\BZ_{\geq 0} + B \subseteq B$, so that $\BG_r^B \subseteq \BG_r^A$ are subgroups, and the smaller one is normal in the bigger one. Put
\[
X_B^A = \{g \in \BG_r^A \colon g^{-1}F(g) \in (\overline \BU_r^{+} \cap F\BU_r^{+}) \cdot \BG_r^B \} / \BG_r^B.
\]
If $A = \widetilde\Phi^+$, $B = \widetilde\Phi^f$ for some $f \in \tPhi^+$, then we write $X_f^+ = X_f^A$. For any character $\chi \colon (\BT_r^{+} \cap \BG_r^A)^F \to \overline\BQ_\ell^\times$, we have the $\chi$-weight complex $R\Gamma_c(X_B^A,\Lambda)[\chi] = R\Gamma_c(X_B^A,\Lambda) \otimes_{\Lambda(\BT_r^{+} \cap \BG_r^A)^F} \chi$. Just as in \cite[\S5.2]{IvanovNie_24} we have the map
\[
\pi_f^{A:B} = u^{-1} \circ \pr_f \circ L \circ s_{A:B}\colon \BG_r^A/\BG_r^B \to \BA_{\CO_f}.
\]

Our first observation is that \cite[Proposition 5.3]{IvanovNie_24} admits the following generalization.

\begin{proposition} \label{prop:cartesian} Let $\widetilde\Phi^r \subseteq C \subseteq B \subseteq A \subseteq \widetilde\Phi^+$ be closed subsets with $A+B \subseteq B$, $A+C \subseteq C$. Let $f \in B$ and suppose that $C = B \sm \CO_f$ and  $A+\CO_f \subseteq C$. Let $q_f \colon X_C^A \to X_B^A$ denote the natural projection. Then the following hold.

(1) Suppose that $f \in \D_\aff^+$. Then the map
\[
\psi = (q_f, \pr_f, \pr_{F^{a_2}(f)}, \dots,\pr_{F^{a_{2b(f)-2}}(f)}): X_C^A \cong X_B^A \times \prod_{i=0}^{b(f)-1} \BA_{f}
\]
is an isomorphism.

(2) If $f \in \BZ_{\ge 1}$ (in which case $\BA_{\CO_f} = \BA_f = V$), then there is a Cartesian diagram \[ \xymatrix{
X_C^A \ar[d]_{q_f} \ar[r]^{\pr_f} &  \BA_f \ar[d]_{-L} \\
X_B^A \ar[r]^{\pi^{A:B}_f} &  \BA_f.} \]
\end{proposition}
\begin{proof}
The proof is the same as in \cite[Proposition 5.3]{IvanovNie_24}, with the only difference that in (1) the map inverse to $\psi$ is given by
\begin{align*}
\phi(g,y_f, &y_{F^{a_2}(f)}(y_f), \dots, y_{F^{a_{2b(f)-2}(f)}(f)}) \\
&=s_{A:B}(g) \prod_{j=0}^{a_2-1}F^j(u(y_f)) \cdot \prod_{j=a_2}^{a_4-1}F^j(u(y_{F^{a_2}(f)})) \cdot \dots \cdot \prod_{j=a_{2b(f)-2}}^{a_{2b(f)}-1}F^j(u(y_{F^{a_{2b(f)-2}}(f)})),
\end{align*}
and instead of \cite[Lemma 5.4]{IvanovNie_24} we use Lemma \ref{lm:5_4}.
\end{proof}
\begin{lemma}\label{lm:5_4}
Let $f \in -\D_\aff^+$ and let $x = (x_i)_{0\leq i < \#\CO_f} \in \BA_{\CO_f}$ with $x_i \in \BA_{F^i(f)}$. Suppose that $L(x) \in \prod_{i=0}^{b(f)-1}\BA_{F^{a_{2i}}(f)}$. Then for each $0\leq j <\#\CO_f$, $x_{F^j(f)} = F^{j-a_{2i}}(x_{F^{a_{2i}}(f)})$ for $a_{2i} \leq j < a_{2i+2}$. In particular, 

(1) $L(x)_{f} = x_{a_{2b(f)-2}}^{q^{a_{2b(f)} - a_{2b(f)-2}}} - x_0$, $L(x)_{a_{2i}} = x_{a_{2i-2}}^{q^{a_{2i} - a_{2i-2}}} - x_{2i}$ for $0< i < b(f)$, and $L(x)_j = 0$ if $j \neq a_{2i}$ for any $i$; 

(2) $x=0$ if and only if $x_{a_{2i}} = 0$ for all $0\leq i<b(f)-1$.
\end{lemma}
\begin{proof}
The proof is a direct computation.
\end{proof}

For a character $\chi \colon (\BT_{f+}^+)^F \to \Lambda^\times$ we denote by $\chi_{f+}^f$ the restriction of $\chi$ to $(\BT_{f+}^f)^F$. As in \cite[Corollary 5.9]{IvanovNie_24}, the previous proposition implies the following.

\begin{corollary} \label{coh}
Let $f \in \tPhi^+$ and let $\chi \colon (\BT_{f+}^+)^F \to \Lambda^\times$ be a character.

(1) If $f \in \Phi_\aff^+$, then $R\Gamma_c(X^+_{f+}, \overline \BQ_\ell)[\chi] \cong R\Gamma_c(X^+_f, \overline \BQ_\ell)[\chi][-2b(f)]$.

(2) If $f \in \BZ_{\ge 1}$, then $R\Gamma_c(X^+_{f+}, \overline \BQ_\ell)[\chi_{f+}^f] \cong R\Gamma_c(X^+_f, \pi^*(\CL_{\chi_{f+}^f}))$, and hence \[R\Gamma_c(X^+_{f+}, \Lambda)[\chi] \cong R\Gamma_c(X^+_f, \pi^*(\CL_{\chi_{f+}^f}))[\chi],\] where we write $\pi = \pi^{\tPhi^+:\tPhi^f}_f$.
\end{corollary}

Write $H_i(X^+,\Lambda)[\chi] = H^{-i}(h_\natural\Lambda[\chi])$, where $h \colon X^+ \to \Spec \overline\BF_q$ is the structure map. As in \cite[Corollary 5.10, \S2.7]{IvanovNie_24}, \Cref{coh} implies that $H_i(X^+,\Lambda)[\chi] = H_c^{2d_r-i}(X^+_r,\Lambda)[\chi]$ for all $r \geq $ the depth of $\chi$, where $d_r$ is the dimension of $X^+_r$. In this way, Theorem \ref{thm:generalization_first_article}(1),(2) reduce to the following results stated in terms of the usual \'etale cohomology with compact support.

\begin{theorem}\label{thm:coho_formulation}
Assume $p$ is not a torsion prime for $G$. Let $f \in \tPhi^+$ and let $\chi \colon (\BT_f^+)^F \to \Lambda^\times$ be a character. Then there exists some $s = s_{f,\chi} \in \BZ_{\geq 0}$ (independent of $\Lambda$) such that
\[ H_c^i (X^+_f,\Lambda)[\chi] \neq 0 \quad \Leftrightarrow \quad i=s. \]
Moreover, $H_c^s (X^+_f,\Lambda)[\chi]$ is a free $\Lambda$-module on which $F^N$ acts through the scalar $(-1)^{s'} q^{sN/2}$ for some $s' \in \BZ$.
\end{theorem}

We will prove this theorem in \S\ref{sec:proof_coho_thm} first in the case when $\Lambda \in \{L,\lambda\}$ and then for $\Lambda = \CO_L$. From now on until the end of \S\ref{sec:jumps} we assume that $\Lambda \in \{L,\lambda\}$. In particular, all cohomology groups are vector spaces over a field and $H^i(X,\Lambda)[\chi]$ is simply the usual (non-derived) $\chi$-weight part $H^i(X,\Lambda) \otimes_{\Lambda(\BT_r^+)^F} \chi$.

\subsection{Reduction to semisimple simply connected case}\label{sec:sssc_case}
Let $\widetilde G \to G$ be the simply connected cover of $G$. Identify the reduced buildings of $G$ and $\widetilde G$ and write $\widetilde G, \widetilde T, \widetilde U, \widetilde \BT_r, \widetilde \BU_r, \widetilde X^+_f, \dots$ for the objects corresponding to $\widetilde G$.

Following \cite[1.24]{DeligneL_76}, if $\alpha \colon A \to B$ is a homomorphism of finite groups and $Y$ is a space (scheme or fpqc-sheaf) on which $A$ acts, we let the \emph{induced space} ${\rm Ind}_A^B Y$ be the (unique up to a unique isomorphism) $B$-space $I$ equipped with an $A$-equivariant map $Y \to I$ such that $\Hom_B(I,V) = \Hom_A(Y,V)$ for any $B$-space $V$. (Minor variation of) the following statement already appears in \cite[proof of Lemma 4.3.3]{DI} without proof. We give the proof here.

\begin{lemma} \label{lm:induced_space}
We have $X^+_f = {\rm Ind}_{(\widetilde \BT_f^+)^F}^{(\BT_f^+)^F} \widetilde X^+_f$.
\end{lemma}

\begin{proof}
The kernel of $\widetilde G \to G$ is contained in the center of $\widetilde G$. Thus the maps $\widetilde \BT^+_f \to \BT^+_f$ and $\widetilde \BG^+_f \to \BG^+_f$ are injective, and we identify $\widetilde \BT^+_f$, $\widetilde \BG^+_f$ with their images. Also, $\BG^+_f/\widetilde \BG^+_f \cong \BT^+_f/\widetilde \BT^+_f$; we denote this group by $C$.

Note that any $g \in X^+_f$ can be written as $g = \tau_1 g_1$ with $g_1 \in \widetilde \BG^+_f$ and $\tau_1 \in \BT^+_f$. Then $g_1^{-1} \tau_1^{-1}F(\tau_1)F(g_1) = g^{-1}F(g) \in \overline \BU^+_f \cap F\BU^+_f\subseteq \widetilde \BG^+$. As $\widetilde \BG_f^+$ is normal in $\BG_f^+$, it follows that $\tau_1 = F(\tau_1) \in C$, i.e., $\tau_1 \in C^F$. But as $\widetilde \BT_f^+$ is connected, we have $(\BT^+)^F / (\widetilde \BT^+)^F = C^F$. Thus, changing $g_1$ and $\tau_1$ if necessary, we may achieve that $\tau_1 \in (\BT^+_f)^F$. But then it is clear that $g_1^{-1}F(g_1) = g^{-1}F(g) \in \overline \BU^+_f \cap F\BU^+_f$, which implies that $g_1 \in \widetilde X^+_f$. Thus $X^+_f \cong \coprod_{\tau \in (\BT_f^+)^F/(\widetilde \BT_f^+)^F} \tau \widetilde X_f^+$, which is precisely the induced space.
\end{proof}

\begin{remark}
The analogous statement holds for $\BG, \BT, X =\{g \in \BG \colon g^{-1}F(g) \in \overline \BU \cap F\BU\} $ instead of $\BG^+, \BT^+,X^+$. There, $\widetilde \BG \to \BG$ can be non-injective, and its kernel equals the \emph{perfection} of $\ker(\widetilde G \to G)$. The situation with cokernels is similar as in the above proof.
\end{remark}

\begin{example}
If $k=\BF_2(\!(\varpi)\!)$, $G = {\rm PGL}_2$, $\bx$ hyperspecial, then $\widetilde G = \SL_2$, the maps $\widetilde\BG \to \BG$, $\widetilde \BT \to \BT$ are injective with cokernels isomorphic to $C = H^1((\Spec \overline\BF_2[\![\varpi]\!])_{\rm fppf}, \mu_2) = {\rm coker}(\overline\BF_2[\![\varpi]\!]^\times \stackrel{(\cdot)^2}{\to} \overline\BF_2[\![\varpi]\!]^\times)$, and $C^F$ is an infinite-dimensional $\BF_2$-vector space.
\end{example}

It follows that if $\chi \colon (\BT_f^+)^F \to \Lambda^\times$ is a character and $\widetilde\chi$ is it's pullback to a character of $(\widetilde \BT_f^+)^F$, then $H_c^i(X^+_f,\Lambda)[\chi] = \bigoplus_{(\BT_f^+)^F/(\widetilde \BT_f^+)^F}H_c^i(\widetilde X^+_f,\Lambda)[\widetilde \chi]$ as vector spaces with Frobenius action. In particular, if Theorem \ref{thm:coho_formulation} holds for $\widetilde X_f^+$, then it holds for $X_f^+$.

\subsection{Handling jumps}\label{sec:jumps}


As in \cite[\S5.6]{IvanovNie_24}, fix a positive integer $h \le r$ and a character $\chi$ of $(\BT_{h+}^+)^F$. Recall that $\BT_{h+}^h \cong \BA_h = V = X_*(T) \otimes \overline \BF_q$, so that for any $M \geq 1$ we have the map
\[
\Nm_M: V \to V, ~ v \mapsto v + F(v) + \cdots + F^{M-1}(v).
\]
Note that for $\alpha \in \Phi$, the subgroup $\Nm_M (\a^\vee \otimes \BF_{q^M})$ is independent of the choice of the integer $M \in \BZ_{\geq 1}$ satisfying $F^M(\alpha) = \alpha$. Using the character $\chi$ we define the $F$-stable subset
\[
\Phi_{\chi} = \{\a \in \Phi; \chi \circ \Nm_N (\a^\vee \otimes \BF_{q^N}) = \{1\}\}.
\]
of $\Phi$. Clearly, $-\Phi_\chi = \Phi_\chi$. However, $\Phi_\chi$ does not need to be closed under addition. This fact as well as (essentially) the proof of the following lemma  was explained to us by David Schwein.

\begin{lemma}\label{lm:closed_addition} If $p>3$ or if $p$ is not a torsion prime for $G$, then $\Phi_\chi$ is closed under addition.
\end{lemma}
\begin{proof}
For $m \in \BZ$ and $\alpha \in \Phi$, put
\[ \chi_{m\alpha} = \chi \circ \Nm_N \circ (\alpha^\vee \otimes \BF_{q^N}) \circ (m\cdot) \colon \BF_{q^N} \stackrel{m}{\to} \BF_{q^N} \to V^{F^N} \to V^F \to \Lambda^\times,\]
where the first map is multiplication by $m$. For $\a \in \Phi$, we have $\alpha^\vee = \frac{2}{(\a,\a)}$, where $(\a,\a) = |\a|^2$ is the square of the length of $\a$. For $\a, \b \in \Phi$, write $n_{\a,\b} = \frac{(\a,\a)}{(\b,\b)}$. By \cite[Chap.VI, \S4, Proposition 12(i)]{Bourbaki_68}, the numbers $n_{\a,\b}$ can only take the values $1,2^{\pm 1}, 3^{\pm 1}$ and by inspection one checks that the values $2^{\pm 1}$ resp. $3^{\pm 1}$ can only appear if $p$ is a torsion prime for $G$. Suppose now that $\alpha,\beta \in \Phi_\chi$, that is $\chi_\a, \chi_\a$ are trivial. Suppose that $\gamma = \alpha + \beta \in \Phi$. Thus $\gamma^\vee = \frac{2}{(\gamma,\gamma)} \gamma = n_{\a,\gamma} \alpha^\vee + n_{\b,\gamma}\beta^\vee$. Let $m \in \BZ_{\geq 1}$ be the smallest positive integer such that $mn_{\a,\gamma}$ and $mn_{\b,\gamma}$ lie in $\BZ$. Then we have $m\gamma^\vee = mn_{\a,\gamma} \a^\vee + mn_{\b,\gamma}\b^\vee$. Note that we have $\chi_{m\gamma} = \chi_{mn_{\a,\gamma} \a} \cdot \chi_{mn_{\b,\gamma} \b}$ as characters of $\BF_{q^N}$. As $\chi_{\a},\chi_\b$ are trivial, also $\chi_{mn_{\a,\gamma} \a}$ and $\chi_{mn_{\b,\gamma} \b}$ are trivial. Thus also $\chi_{m\gamma}$ is trivial and as $m$ invertible in $\BF_q$ by assumption, it follows that $\chi_{\gamma} = 1$, that is $\gamma \in \Phi_\chi$.
\end{proof}

Let $M = M_\chi$ be the subgroup generated by $T$ and $U_\a$ for $\a \in \Phi_\chi$. By Lemma \ref{lm:closed_addition}, $M$ is reductive with root system $\Phi_M=\Phi_\chi$. First we note that \Cref{prop:BW_modular} implies the following generalization of \cite[Proposition 5.16]{IvanovNie_24} for arbitrary $\Lambda \in \{L,\CO_L,\lambda\}$.

\begin{corollary}\label{computation}
     Let $\a \in \Phi \setminus \Phi_M$ and let $\k: \BG_a \to V$ be the map given by $x \mapsto \a^\vee \otimes x$ for $x \in \overline \BF_q$.

     (1) $\k^* \CL_{\chi_{h+}^h}$ is nontrivial and hence $H_c^i(\BG_a, \k^* \CL_{\chi_{h+}^h}) = 0$ for $i \in \BZ$;

     (2) if $N$ is even and $F^{N/2}(\a) = -\a$, then \begin{align*} \dim H_c^i(\BG_a, \tau^* \CL_{\chi_{h+}^h}) = \begin{cases} q^{N/2} &\text{ if } i = 1; \\ 0, &\text{otherwise,} \end{cases} \end{align*} where $\tau: \BG_a \to V$ is given by $x \mapsto \a^\vee \otimes x^{q^{N/2} + 1}$. Moreover, in this case $F^N$ acts on $H_c^1(\BG_a, \tau^* \CL_{\chi_{h+}^h})$ via $-q^{N/2}$. 
\end{corollary}
\begin{proof}
The proof of \cite[Proposition 5.16]{IvanovNie_24}(1) generalizes to arbitrary $\Lambda$ and \cite[Proposition 5.16]{IvanovNie_24}(2) follows from \Cref{prop:BW_modular}.
\end{proof}

Let $\tPhi_M \subseteq \widetilde \Phi$ be the set of affine roots of $M$. Consider
\[D = (\Phi_\aff^+ \cap \Phi_h^+) \setminus \tPhi_M = \{f \in \Phi_\aff^+ \setminus \tPhi_M; f < h\}.\]
As $\tPhi_M$ is $F$-stable, $D$ is a union of $F$-orbits in $\widetilde\Phi$. Similar as in \cite[\S5.6]{IvanovNie_24}, we can number the $F$-orbits of $D$ as
\[
\CO_1, \dots, \CO_{m-1}, \CO_m = \CO_m^\flat, \dots, \CO_n = \CO_n^\flat, \CO_{m-1}^\flat, \dots, \CO_1^\flat
\]
where $\CO^\flat = \{h-f \colon f \in \CO\}$, and such that
\[\CO_1(\bx) \le \cdots \le \CO_{m-1}(\bx) \le \frac{h}{2} = \CO_m(\bx) = \cdots = \CO_n(\bx) = \frac{h}{2} \le \CO_{m-1}^\flat(\bx) \le \cdots \le \CO_1^\flat(\bx),\]
$\CO_i < \CO_i^\flat$ for $1 \le i \le m-1$ and $\CO_{m-1}^\flat < \cdots < \CO_1^\flat$. Define $N_i = \#\CO_i$.

Set $D_i^\flat = \bigcup_{j=1}^i \CO_j^\flat$ for $1 \le i \le m-1$, and $D_m^\flat = \bigcup_{j=1}^n \CO_j^\flat$. Define
\[A_i = \tPhi^+ \setminus \bigcup_{j=1}^{i-1} \CO_j, \quad B_i = \tPhi^h \cup D_i^\flat, \quad C_{i-1} = B_{i-1} \setminus \{h\}.\] Moreover, we set $A_0 = A_1 = \tPhi^+$, $B_0 = \tPhi^h$ and $C_0 = B_0 \setminus \{h\}$. Note that $A_m = B_m \cup \tPhi_M^+$, where $\tPhi_M^+ = \tPhi_M \cap \tPhi^+$.

Let $g \in \BG_r^+$, $x \in \BA[r]$ and $E \subseteq \tPhi_r^+$. As in \cite[\S5.6]{IvanovNie_24} we set $g_E = \pr_E(g) \in u(\BA_E)$, $x_E = p_E(x) \in \BA_E$ and $\hat x = u(x) \in \BG_r^+$. For $f \in \tPhi_r^+$ we will set $x_f = x_{\{f\}}$ and $x_{\ge f} = x_{\tPhi^f}$. We can define $g_f$ and $g_{\ge f} \in \BG_r^+$ in a similar way. We identify $g_f \in u(\BA_f)$ with $u\i(g_f) \in \BA_f$ according to the context.

Note that \cite[Lemmas 5.12, 5.13 and 5.14]{IvanovNie_24} hold in our more general setup without any change and with literally the same proofs. (Note that the proof of \cite[Lemmas 5.12]{IvanovNie_24} uses the following property of $\Phi_M \subseteq \Phi$: if $\alpha \in \Phi_M$, $\beta \in \Phi \sm \Phi_M$, then $\alpha+\beta \not\in \Phi_M$. This holds when $\Phi_M \subseteq \Phi_G$ is a Levi subsystem, which follows from $p$ not being a torsion prime for $G$ by \cite[Lemma 3.6.1]{Kaletha_19}. This is guaranteed by the assumptions in \Cref{thm:coho_formulation}.)

We now generalize \cite[Proposition 5.15]{IvanovNie_24}. Set $\pi = \pi_h^{\tPhi^+: \tPhi^h}: \BG_h^+ = \BG_r^+/\BG_r^h \to \BA_h \cong V$.

\begin{proposition} \label{prop:factor}
    Let $1 \le i \le m$. Then there is an isomorphism \[\psi_i: X^{A_i}_h \cong X^{A_i}_{B_i} \times \BA_{D_i^\flat \cap -\tD^+}.\] Moreover, for $(\hat x, y) \in X^{A_i}_{B_i} \times \BA_{D_i^\flat \cap -\tD^+}$ with $x \in \BA_{A_i \setminus B_i}$ we have

    (1) Assume that $1 \le i \le m-1$, fix some $f \in \CO_i \cap -\tD^+$ and let $f^\flat = h-f \in \CO_i^\flat$. With notation of \eqref{eq:sequence_of_sign_changes}, we have \[\pi(\psi_i\i(\hat x, y)) = \sum_{j=0}^{b(f)-1} \a_{F^{a_{2j}}(f)}^\vee \otimes (x_{F^{a_{2j+1}}(f)}^{q^{a_{2j+1} - a_{2j-1}}} - x_{F^{a_{2j-1}}(f)}) y_{F^{2j}(f^\flat)}^{q^{a_{2j}-a_{2j-1}}} + \pi(\psi_i\i(\hat x, 0)) \in V.\]

    (2) Assume that $i=m$. For each $m \leq k \leq n$ fix some $f_k \in \CO_k \cap -\tD^+$. Then $\pi(\psi_m\i(\hat x, y))$ equals $ \pi(\psi_m\i(\hat x, 0))$ plus the sum over $m \leq k \leq n$ of the following term corresponding to $f=f_k$ (where $a_j = a_j(f_k)$ and $b = b(f_k)$ are as in \eqref{eq:sequence_of_sign_changes}):
    \begin{align*} - \alpha^\vee_{f} \otimes y_{f} \cdot y_{F^{a_{b-1}}(f)}^{q^{a_b - a_{b-1}}} + \sum_{j=1}^{(b-1)/2} \a_{F^{a_{2j}}(f)}^\vee \otimes (y_{F^{a_{2j}}(f)} - y_{F^{a_{2j-2}}(f)}^{q^{a_{2j} - a_{2j-2}}})y_{F^{a_{2j-1+b}}(f)}^{q^{a_{2j+b} - a_{2j-1+b}}} .\end{align*}
\end{proposition}

Note that the formulas in Proposition \ref{prop:factor} are similar to those in \cite[proof of Lemma 6.10]{Nie_24}.
We deduce the generalization of \cite[Proposition 5.17]{IvanovNie_24} from this.

\begin{proposition} \label{key}
    Write $X^M_h = X_h \cap \BM_h^+$, and let $\pi_M$ be the restriction of $\pi$ to $\BM_h^+$. The following statements hold:

         (1) $H_c^j(X^{A_i}_h, \pi^* \CL_{\chi_{h+}^h}) \cong H_c^j(X^{A_{i+1}}_h, \pi^* \CL_{\chi_{h+}^h})^{\oplus q^{M_i}}$ for $1 \le i \le m-1$ and some $M_i \in \BZ_{\geq 1}$;

         (2) $H_c^j(X^{A_m}_h, \pi^* \CL_{\chi_{h+}^h}) \cong H_c^{j-n-m+1}(X^M_h, \pi_M^* \CL_{\chi_{h+}^h})^{\oplus q^{C_1}} ((-1)^{C_2}q^{C_3})$,
         where $C_1 = \sum_{i=m}^n q^{\#O_i/2}$ and $C_2 = \sum_{i=m}^n N/\#O_i$ and $C_3 = q^{N(n+m-1)/2}$.
\end{proposition}
Note that in general $C_2 \neq C_3$, in contrast to the special case of \emph{loc.~cit.}
\begin{proof}
We can proceed exactly as in \cite[proof of Proposition 5.17]{IvanovNie_24}, by noting that by Proposition \ref{prop:factor} the local system $\pi^\ast\CL_{\chi_{h+}^h}$ is trivial on a fiber over $\hat x \in X_{B_i}^{A_i}$ if and only if $x_{F^{a_{2j+1}}(f)}^{q^{a_{2j+1} - a_{2j-1}}} - x_{F^{a_{2j-1}}(f)} = 0$ for each $j$, which shows that for $0\leq i \leq m-1$,
\[
H_c^a(X_h^{A_i},\pi^\ast\CL) \cong H_c^a(Y_h^{A_{i+1}},\pi^\ast\CL)^{\oplus q^{M_i}}
\]
for some $M_i \in \BZ_{\geq 1}$, and similarly for part (2). In course of proving (2), when repeating the computation from \emph{loc.~cit.}, we obtain
\begin{equation}\label{eq:coho_of_YhAm}
H^j_c(X_h^{A_m},\pi^\ast\CL) \cong \otimes_{i=m}^n H_c^1(\BA_{f_i},\tau_i^\ast\CL) \otimes \otimes_{i=1}^{m-1}H_c^2(\BA_{f_i^\flat},\Lambda) \otimes H_c^{j-n-m+1}(X_h^M,\pi^\ast\CL).
\end{equation}
Whereas for $i\leq m-1$, $F^N$ acts on each $H_c^2(\BA_{f_i^\flat},\Lambda)$ by $q^N$ as in \emph{loc.~cit.}, we have that $F^{\#O_i}$ acts on $H_c^1(\BA_{f_i},\tau_i^\ast\CL)$ by $-q^{\#O_i/2}$ (by \cite[Proposition 5.16(2)]{IvanovNie_24}). Thus $F^N$ acts on this space by $(-q^{\#O_i/2})^{N/\#O_i} = (-1)^{N/\#O_i}q^{N/2}$. Altogether, we see that \eqref{eq:coho_of_YhAm} equals
\[
H_c^{j-n-m+1}(X_h^M,\pi^\ast\CL)^{\oplus\sum_{i=m}^n q^{\#O_i/2}}((-1)^{\sum_{i=m}^n N/\#O_i} q^{N(n+m-1)/2}),
\]
where the number of summands again follows from \cite[Proposition 5.16(2)]{IvanovNie_24}.
\end{proof}

\subsection{Proof of \Cref{thm:coho_formulation}}\label{sec:proof_coho_thm}
We may assume by \S\ref{sec:sssc_case} that $G$ is semisimple and simply connected. Then \cite[Lemma 2.2]{IvanovNie_24} guarantees that $M_\chi \neq G$, where $M_\chi$ is as in the beginning of \S\ref{sec:jumps}. Then exactly the same induction procedure as in \cite[\S5.7]{IvanovNie_24}, exploiting the results of \S\ref{sec:jumps}, finishes the proof when $\Lambda \in \{L,\lambda\}$. 

Suppose now that $\Lambda = \CO_L$. Let $\bar\chi \colon (\BT_r^+)^F \to \lambda^\times$, $\chi[\ell^{-1}] \colon (\BT_r^+)^F \to L^\times$ denote the characters induced by $\chi$. As $\ker(\CO_L^\times \to \lambda^\times) = \mu_{\ell^\infty}(L)$ intersects the image of $\chi$ trivially, the sequence of jumps and of the root subsystems $\Phi_{\bar\chi}$ and $\Phi_{\chi[\ell^{-1}]}$ attached to $\bar\chi$ and $\chi[\ell^{-1}]$ at the beginning of \S\ref{sec:jumps}, coincide (with other words, the twisted Levi subgroups and the sequences of depths, attached to $\bar\chi$ and $\chi[\ell^{-1}]$ through the Howe factorizations, coincide). It follows that the induction procedures in the proof of \Cref{thm:coho_formulation} in the two cases $\Lambda = L$ and $\Lambda = \lambda$ run exactly parallel to each other, producing the same non-vanishing degree $s_{\chi,r}$ and cohomology groups of the same dimension. Thus $s_{\bar\chi,r} = s_{\chi[\ell^{-1}],r}$ and $\dim_\lambda H^{s_{\bar\chi,r}}(X,\lambda)[\bar\chi] = \dim_\lambda H^{s_{\chi[\ell^{-1}],r}}(X,L)[\chi[\ell^{-1}]]$. Now we can conclude by \Cref{lm:integral_coefficients}(2) and (3).


\bibliography{bib_ADLV}{}
\bibliographystyle{amsalpha}

\end{document}